\documentclass[10pt]{amsart}

\usepackage{microtype}

\usepackage{amsmath}
\usepackage{amsmath}
\usepackage{amssymb}
\usepackage{tikz}
\usepackage{color}

\usepackage{microtype}
\usepackage{amsmath}
\usepackage{amsmath}
\usepackage{amssymb}
\usepackage{tikz}
\usepackage{color}

\newcommand{\e}{\mathsf e}

\newcommand\TTTT{%
 \textsf{T\kern-0.15em\raisebox{-0.55ex}T\kern-0.15emT\kern-0.15em\raisebox{-0.55ex}2}%
}

\providecommand{\U}[1]{\protect\rule{.1in}{.1in}}

\usepackage{graphicx}

\usepackage{epsfig}
\usepackage{bm}
\usepackage{prooftree}

\title{
 On noncommutative generalisations of Boolean algebras}

\author[A. Bucciarelli]{A. Bucciarelli}\address{A. Bucciarelli, Universit\'e Paris Diderot}\email{buccia@irif.fr}
\author[A. Salibra]{A. Salibra}\address{A. Salibra, Universit\`a Ca'Foscari Venezia}\email{salibra@unive.it}
\date{\today}


\begin{document}
\dedicatory{\large{In honor of Jonathan Leech}}
\maketitle
\theoremstyle{plain}
\newtheorem{theorem}{Theorem}[section]
\newtheorem{lemma}[theorem]{Lemma}
\newtheorem{corollary}[theorem]{Corollary}
\newtheorem{proposition}[theorem]{Proposition}
\newtheorem{claim}[theorem]{Claim}
\theoremstyle{remark}
\newtheorem{remark}{Remark}
\newtheorem{algorithm}{Algorithm}
\newtheorem{axiom}{Axiom}
\newtheorem{case}{Case}
\newtheorem{conclusion}{Conclusion}
\newtheorem{condition}{Condition}
\newtheorem{conjecture}{Conjecture}
\theoremstyle{definition}
\newtheorem{definition}{Definition}[section]
\newtheorem{notation}[definition]{Notation}
\newtheorem{criterion}[definition]{Criterion}
\newtheorem{example}{Example}
\newtheorem{problem}{Problem}
\newtheorem{exercise}{Exercise}
\newtheorem{solution}{Solution}
\newtheorem{acknowledgement}{Acknowledgement}
\newtheorem{summary}{Summary}

\newcommand{\FL}{{\text{FL}}}
\newcommand{\CRL}{{\text{CRL}}}
\newcommand{\RL}{{\text{RL}}}
\newcommand{\oor}{{\overline{\vee}}}

\newcommand{\impl}{\Rightarrow}
\newcommand{\Ra}{\Rightarrow}
\newcommand{\La}{\Leftarrow}
\newcommand{\rsa}{\rightsquigarrow}
\newcommand{\lsa}{\leftsquigarrow}
\newcommand{\thra}{\twoheadrightarrow}
\newcommand{\thla}{\twoheadleftarrow}
\newcommand{\rest}{\upharpoonleft}
\newcommand{\bigsupr}{\bigsqcup}
\newcommand{\biginfr}{\bigsqcap}
\newcommand{\supr}{\sqcup}
\newcommand{\infr}{\sqcap}
\newcommand{\sqleq}{\sqsubseteq}
\newcommand{\sub}{\subseteq}
\newcommand{\emp}{\emptyset}
\newcommand{\lam}{\mathbb{\leftthreetimes}}
\newcommand{\lint}{\llbracket}
\newcommand{\rint}{\rrbracket}

\newcommand{\intx}{\lint x \rint_{\rho}}
\newcommand{\intca}{\lint c_a \rint_{\rho}}
\newcommand{\intpq}{\lint PQ \rint_{\rho}}
\newcommand{\intmn}{\lint MN \rint_{\rho}}
\newcommand{\intm}{\lint M \rint_{\rho}}
\newcommand{\intn}{\lint N \rint_{\rho}}
\newcommand{\intp}{\lint P \rint_{\rho}}
\newcommand{\intq}{\lint Q \rint_{\rho}}
\newcommand{\intlxp}{\lint \lambda x.P \rint_{\rho}}
\newcommand{\intlxm}{\lint \lambda x.M \rint_{\rho}}
\newcommand{\intlxn}{\lint \lambda x.N \rint_{\rho}}
\newcommand{\intlyp}{\lint \lambda y.P \rint_{\rho}}
\newcommand{\last}{\lambda^*x}
\newcommand{\lx}{\lambda^*x}
\newcommand{\ly}{\lambda^*y}
\newcommand{\lxn}{\lambda^*x_1\ldots x_n}

\newcommand{\mfA}{\mathfrak{A}}
\newcommand{\mfB}{\mathfrak{B}}
\newcommand{\mfC}{\mathfrak{C}}
\newcommand{\mfD}{\mathfrak{D}}
\newcommand{\mfE}{\mathfrak{E}}
\newcommand{\mfF}{\mathfrak{F}}
\newcommand{\mfG}{\mathfrak{G}}
\newcommand{\mfH}{\mathfrak{H}}
\newcommand{\mfI}{\mathfrak{I}}
\newcommand{\mfJ}{\mathfrak{J}}
\newcommand{\mfK}{\mathfrak{K}}
\newcommand{\mfL}{\mathfrak{L}}
\newcommand{\mfM}{\mathfrak{M}}
\newcommand{\mfN}{\mathfrak{N}}
\newcommand{\mfO}{\mathfrak{O}}
\newcommand{\mfP}{\mathfrak{P}}
\newcommand{\mfQ}{\mathfrak{Q}}
\newcommand{\mfR}{\mathfrak{R}}
\newcommand{\mfS}{\mathfrak{S}}
\newcommand{\mfT}{\mathfrak{T}}
\newcommand{\mfU}{\mathfrak{U}}
\newcommand{\mfV}{\mathfrak{V}}
\newcommand{\mfW}{\mathfrak{W}}
\newcommand{\mfX}{\mathfrak{X}}
\newcommand{\mfY}{\mathfrak{Y}}
\newcommand{\mfZ}{\mathfrak{Z}}

\newcommand{\mcA}{\mathcal{A}}
\newcommand{\mcB}{\mathcal{B}}
\newcommand{\mcC}{\mathcal{C}}
\newcommand{\mcD}{\mathcal{D}}
\newcommand{\mcE}{\mathcal{E}}
\newcommand{\mcF}{\mathcal{F}}
\newcommand{\mcG}{\mathcal{G}}
\newcommand{\mcH}{\mathcal{H}}
\newcommand{\mcI}{\mathcal{I}}
\newcommand{\mcJ}{\mathcal{J}}
\newcommand{\mcK}{\mathcal{K}}
\newcommand{\lang}{\mathcal{L}}
\newcommand{\mcL}{\mathcal{L}}
\newcommand{\mcM}{\mathcal{M}}
\newcommand{\mcN}{\mathcal{N}}
\newcommand{\mcO}{\mathcal{O}}
\newcommand{\mcP}{\mathcal{P}}
\newcommand{\mcQ}{\mathcal{Q}}
\newcommand{\mcR}{\mathcal{R}}
\newcommand{\mcS}{\mathcal{S}}
\newcommand{\mcT}{\mathcal{T}}
\newcommand{\mcU}{\mathcal{U}}
\newcommand{\mcV}{\mathcal{V}}
\newcommand{\mcW}{\mathcal{W}}
\newcommand{\mcX}{\mathcal{X}}
\newcommand{\mcY}{\mathcal{Y}}
\newcommand{\mcZ}{\mathcal{Z}}

\newcommand{\mcp}{\mathcal{a}}
\newcommand{\mcq}{\mathcal{q}}
\newcommand{\mcr}{\mathcal{r}}
\newcommand{\mcs}{\mathcal{s}}

\newcommand{\mbA}{\mathbf{A}}
\newcommand{\mbB}{\mathbf{B}}
\newcommand{\mbC}{\mathbf{C}}
\newcommand{\mbD}{\mathbf{D}}
\newcommand{\mbE}{\mathbf{E}}
\newcommand{\mbF}{\mathbf{F}}
\newcommand{\mbG}{\mathbf{G}}
\newcommand{\mbH}{\mathbf{H}}
\newcommand{\mbI}{\mathbf{I}}
\newcommand{\mbJ}{\mathbf{J}}
\newcommand{\mbK}{\mathbf{K}}
\newcommand{\mbL}{\mathbf{L}}
\newcommand{\mbM}{\mathbf{M}}
\newcommand{\mbN}{\mathbf{N}}
\newcommand{\mbO}{\mathbf{O}}
\newcommand{\mbP}{\mathbf{P}}
\newcommand{\mbQ}{\mathbf{Q}}
\newcommand{\mbR}{\mathbf{R}}
\newcommand{\mbS}{\mathbf{S}}
\newcommand{\mbT}{\mathbf{T}}
\newcommand{\mbU}{\mathbf{U}}
\newcommand{\mbV}{\mathbf{V}}
\newcommand{\mbW}{\mathbf{W}}
\newcommand{\mbX}{\mathbf{X}}
\newcommand{\mbY}{\mathbf{Y}}
\newcommand{\mbZ}{\mathbf{Z}}

\newcommand{\mbbA}{\mathbb{A}}
\newcommand{\mbbB}{\mathbb{B}}
\newcommand{\mbbC}{\mathbb{C}}
\newcommand{\mbbD}{\mathbb{D}}
\newcommand{\mbbE}{\mathbb{E}}
\newcommand{\mbbF}{\mathbb{F}}
\newcommand{\mbbG}{\mathbb{G}}
\newcommand{\mbbH}{\mathbb{H}}
\newcommand{\mbbI}{\mathbb{I}}
\newcommand{\mbbJ}{\mathbb{J}}
\newcommand{\mbbK}{\mathbb{K}}
\newcommand{\mbbL}{\mathbb{L}}
\newcommand{\mbbM}{\mathbb{M}}
\newcommand{\mbbN}{\mathbb{N}}
\newcommand{\mbbO}{\mathbb{O}}
\newcommand{\mbbP}{\mathbb{P}}
\newcommand{\mbbQ}{\mathbb{Q}}
\newcommand{\mbbR}{\mathbb{R}}
\newcommand{\mbbS}{\mathbb{S}}
\newcommand{\mbbT}{\mathbb{T}}
\newcommand{\mbbU}{\mathbb{U}}
\newcommand{\mbbV}{\mathbb{V}}
\newcommand{\mbbW}{\mathbb{W}}
\newcommand{\mbbX}{\mathbb{X}}
\newcommand{\mbbY}{\mathbb{Y}}
\newcommand{\mbbZ}{\mathbb{Z}}

\newcommand{\mba}{\mathbf{a}}
\newcommand{\mbb}{\mathbf{b}}
\newcommand{\mbc}{\mathbf{c}}
\newcommand{\mbd}{\mathbf{d}}
\newcommand{\mbe}{\mathbf{e}}
\newcommand{\mbf}{\mathbf{f}}
\newcommand{\mbg}{\mathbf{g}}
\newcommand{\mbh}{\mathbf{h}}
\newcommand{\mbi}{\mathbf{i}}
\newcommand{\mbj}{\mathbf{j}}
\newcommand{\mbk}{\mathbf{k}}
\newcommand{\mbl}{\mathbf{l}}
\newcommand{\mbm}{\mathbf{m}}
\newcommand{\mbn}{\mathbf{n}}
\newcommand{\mbo}{\mathbf{o}}
\newcommand{\mbp}{\mathbf{p}}
\newcommand{\mbq}{\mathbf{q}}
\newcommand{\mbr}{\mathbf{r}}
\newcommand{\mbs}{\mathbf{s}}
\newcommand{\mbt}{\mathbf{t}}
\newcommand{\mbu}{\mathbf{u}}
\newcommand{\mbv}{\mathbf{v}}
\newcommand{\mbw}{\mathbf{w}}
\newcommand{\mbx}{\mathbf{x}}
\newcommand{\mby}{\mathbf{y}}
\newcommand{\mbz}{\mathbf{z}}

\newcommand{\msA}{\mathscr{A}}
\newcommand{\msB}{\mathscr{B}}
\newcommand{\msC}{\mathscr{C}}
\newcommand{\msD}{\mathscr{D}}
\newcommand{\msE}{\mathscr{E}}
\newcommand{\msF}{\mathscr{F}}
\newcommand{\msG}{\mathscr{G}}
\newcommand{\msH}{\mathscr{H}}
\newcommand{\msI}{\mathscr{I}}
\newcommand{\msJ}{\mathscr{J}}
\newcommand{\msK}{\mathscr{K}}
\newcommand{\msL}{\mathscr{L}}
\newcommand{\msM}{\mathscr{M}}
\newcommand{\msN}{\mathscr{N}}
\newcommand{\msO}{\mathscr{O}}
\newcommand{\msP}{\mathscr{P}}
\newcommand{\msQ}{\mathscr{Q}}
\newcommand{\msR}{\mathscr{R}}
\newcommand{\msS}{\mathscr{S}}
\newcommand{\msT}{\mathscr{T}}
\newcommand{\msU}{\mathscr{U}}
\newcommand{\msV}{\mathscr{V}}
\newcommand{\msW}{\mathscr{W}}
\newcommand{\msX}{\mathscr{X}}
\newcommand{\msY}{\mathscr{Y}}
\newcommand{\msZ}{\mathscr{Z}}

\newcommand{\nat}{\mathbb{N}}
\newcommand{\stdmod}{\underline{\underline{\mathbb{N}}}}
\newcommand{\rea}{\mathbb{R}}
\newcommand{\intg}{\mathbb{Z}}

\newcommand{\Po}{P\omega}
\newcommand{\ConA}{Con \ \mathbf{A}}
\newcommand{\Alg}{\mathbf{A}}
\newcommand{\KFree}{\mathbf{\overline{F}_K(\overline{X})}}
\newcommand{\AbFree}{\mathbf{\overline{F}(X)}}
\newcommand{\Tsig}{\mathbf{T_{\Sigma}(X)}}
\newcommand{\LT}{\mathbf{\lambda T}}
\newcommand{\TP}{\mathbf{T(P)}}
\newcommand{\FP}{\mathbf{F(P)}}
\newcommand{\TM}{\mathbf{\Lambda / \mcT}}
\newcommand{\TSV}{\mathbf{T_\Sigma(V)}}

\newcommand*{\sqsubsim}{%
\mathrel{\vcenter{\offinterlineskip 
\hbox{$\sqsubset$}\vskip-.130ex\hbox{$\sim$}}}}

\newcommand{\sso}{{\mathbf{1}}}
\newcommand{\osso}{{\overline{\mathbf{1}}}}
\newcommand{\ssz}{{\mathbf{0}}}
\newcommand{\ossz}{{\overline{\mathbf{0}}}}

\newcommand{\parr}{{\bindnasrepma}}
\newcommand{\tensor}{{\varotimes}}
\newcommand{\with}{{\binampersand}}
\newcommand{\plus}{{\oplus}}
\newcommand{\mm}{{\multimap}}

\newcommand{\ot}{\leftarrow}

\newcommand{\ga}{\alpha}
\newcommand{\gb}{\beta}
\newcommand{\gc}{\gamma}
\newcommand{\gd}{\delta}
\newcommand{\gep}{\varepsilon}
\newcommand{\gz}{\zeta}
\newcommand{\geta}{\eta}
\newcommand{\gth}{\theta}
\newcommand{\gi}{\iota}
\newcommand{\gv}{\nu}
\newcommand{\gk}{\kappa}
\newcommand{\gl}{\lambda}
\newcommand{\gm}{\mu}
\newcommand{\gn}{\nu}
\newcommand{\gx}{\xi}
\newcommand{\gp}{\pi}
\newcommand{\gr}{\rho}
\newcommand{\gs}{\sigma}
\newcommand{\gt}{\tau}
\newcommand{\gu}{\upsilon}
\newcommand{\gph}{\varphi}
\newcommand{\gch}{\chi}
\newcommand{\gps}{\psi}
\newcommand{\go}{\omega}
\newcommand{\gS}{\Sigma}
\newcommand{\gL}{\Lambda}

\newcommand{\no}{{\sim\!}}
\newcommand{\la}{{\langle}}
\newcommand{\ra}{{\rangle}}

\newcommand\oa{{\overline{a}}}
\newcommand\ob{{\overline{b}}}
\newcommand\oc{{\overline{c}}}
\newcommand\od{{\overline{d}}}
\newcommand\ox{{\overline{x}}}
\newcommand\oy{{\overline{y}}}
\newcommand\ou{{\overline{u}}}
\newcommand\ov{{\overline{v}}}
\newcommand\ow{{\overline{w}}}
\newcommand\oz{{\overline{z}}}
\newcommand\oZ{{\overline{Z}}}
\newcommand\oQ{{\overline{Q}}}
\newcommand\oN{{\overline{N}}}
\newcommand\oM{{\overline{M}}}
\newcommand\oR{{\overline{R}}}
\newcommand\oX{{\overline{X}}}
\newcommand\ogx{{\overline{\gx}}}

\begin{abstract}
Skew Boolean algebras (skew BA) and  Boolean-like algebras ($n$BA) are one-pointed and $n$-pointed noncommutative
  generalisation of Boolean algebras, respectively. We show that 
  any $n$BA is a cluster of $n$ isomorphic right-handed skew BAs, axiomatised here as the variety of skew star algebras. 
The variety of skew star algebras is shown to be term equivalent to the variety of $n$BAs. 
We use  skew BAs in order to develop
a general theory of multideals for $n$BAs.
We also provide a representation theorem for right-handed skew BAs in terms of $n$BAs of $n$-partitions.
\end{abstract}

\section{Introduction}

Boolean algebras are the main example of a well-behaved double-pointed variety - meaning a variety $\mathcal V$ whose type includes two distinct constants $0, 1$ in every nontrivial $A\in \mathcal V$.
Since there are other double-pointed varieties of algebras that have Boolean-like features, in \cite{first,LPS13}   the notion of Boolean-like algebra (of dimension $2$) was introduced as a generalisation of Boolean algebras to a double-pointed but otherwise arbitrary similarity type. The idea behind this approach was that a Boolean-like algebra of dimension $2$ is an algebra $A$ such that every $a\in A$ is $2$-central in the sense of Vaggione \cite{vaggione}, meaning that $\theta(a, 0)$ and $\theta(a, 1)$ are complementary factor congruences of $A$. Central elements can be given an equational characterisation through the ternary operator $q$ satisfying the fundamental properties of the if-then-else connective. It also turns out that some important properties of Boolean algebras are shared  by $n$-pointed algebras whose elements satisfy all the equational conditions of $n$-central elements through an operator $q$ of arity $n+1$ satisfying the fundamental properties of a generalised if-then-else connective. These algebras, and the varieties they form, were termed {\em Boolean-like algebras of dimension $n$} ($n$BA, for short) in \cite{BLPS18}.

Varieties of $n\mathrm{BA}$s share many remarkable
properties with the variety of Boolean algebras. In particular, any variety of  $n\mathrm{BA}$s 
 is generated by the $n\mathrm{BA}$s of finite
cardinality $n$.  In the pure case (i.e., when the type includes just the generalised
if-then-else $q$ and the $n$ constants $\e_1,\dots,\e_n$), the variety is generated by a unique
algebra $\mathbf{n}$ of universe $\{\e_1,\dots,\e_n\}$, 
so that any pure $n\mathrm{BA}$ is, up to isomorphism, a subalgebra  of $\mathbf{n}^I$, for a suitable set $I$. Another remarkable property of the $2$-element Boolean algebra is the definability of all finite Boolean functions in terms of the connectives {\sc and, or, not}. This property is inherited by the algebra $\mathbf{n}$:  all finite functions on $\{\e_1,\dots,\e_n\}$ are term-definable, so that the variety of pure $n\mathrm{BA}$s is primal. More generally, a variety of an arbitrary type with one generator  is primal if and only if it is a variety of $n\mathrm{BA}$s.
Just like Boolean algebras are the algebraic
counterpart of classical logic $\mathrm{CL}$,
 $n\mathrm{BA}$s are the algebraic
counterpart of a logic $n\mathrm{CL}$.
The complete symmetry of the truth values ${\e}_1,\dots,{\e}_n$, supports the idea that $n\mathrm{CL}$ is the right generalisation of classical logic from dimension $2$ to dimension $n$.
Any tabular logic with $n$ truth values can be conservatively embedded into  $n\mathrm{CL}$.
All the  results described above are shown in \cite{BLPS18}.

Lattices and boolean algebras have been generalized in other directions:
in the last decades
weakenings of lattices where the meet and join operations may fail to be commutative have attracted  the attention of various researchers.
A noncommutative generalisation of lattices, probably the most
interesting and successful, is the concept of \emph{skew lattice }\cite{Lee89} along with the related notion of \emph{skew Boolean algebra} (the interested reader is
referred to \cite{Lee,BL,Leech} or \cite{Spi} for a comprehensive account).
Skew Boolean algebras are non-commutative one-pointed generalisations of Boolean algebras.
The significance of skew Boolean algebras is revealed by a result of Leech \cite{Leech}, stating that any right-handed skew Boolean algebra can be embedded into  some skew Boolean algebra of partial functions. This result has been revisited and further explored in \cite{BKCV} and \cite{Kudr}, showing  that any skew Boolean algebra is dual to a sheaf over a locally-compact Boolean space.  
Skew Boolean algebras are also closely related to discriminator varieties (see \cite{BL,CS15} for the one-pointed case and \cite{first} for the double-pointed one).  


In this paper we establish connections between skew Boolean algebras and 
 Boolean-like algebras of dimension $n$.
We prove that any  $n$BA $\mathbf A$ contains a symmetric cluster of right-handed skew Boolean algebras $S_{1}(\mathbf A),\dots, S_{n}(\mathbf A)$, called its {\em skew reducts}.  
Interestingly, every permutation $\gs$ of the symmetric group $S_n$ determines a bunch of isomorphisms
$$\mathrm{S}_1(\mathbf A) \cong \mathrm{S}_{\gs 1}(\mathbf A)\qquad\dots\qquad \mathrm{S}_n(\mathbf A) \cong \mathrm{S}_{\gs n}(\mathbf A)$$
which shows the inner symmetry of the $n$BAs.
The skew reducts of a $n$BA are so deeply correlated that they allow to recover the full structure of the $n$BA. Then we  introduce a new variety of algebras, called {\em skew star} algebras, equationally axiomatising 
$n$ skew Boolean algebras and their relationships, and we prove that it is term equivalent to the variety of $n$BAs.
We also provide a representation theorem for right-handed skew Boolean algebras in terms of $n$BAs of $n$-partitions.

The notion of ideal plays an important role
in order theory and universal algebra. Ideals, filters
and congruences are interdefinable in Boolean algebras.
In the case of $n$BAs, the couple ideal-filter  is replaced by {\em multideals},
which are tuples 
$(I_1,\ldots,I_n)$ of disjoint skew Boolean ideals satisfying some compatibility conditions that extend in a conservative way those of the Boolean case.
We show that there exists a bijective correspondence between multideals and
congruences on  $n$BAs,  rephrasing the well known correspondence of the Boolean case.
The proof of this result makes an essential use of the notion of
a coordinate, originally defined in  \cite{BLPS18} and rephrased here in terms of the  operations of the skew reducts. Any element $x$ of a $n$BA $\mathbf A$ univocally determines a $n$-tuple
of elements of the canonical inner Boolean algebra $B$ of $\mathbf A$, its coordinates,
 codifying $x$  as a ``linear combination''.
In the Boolean case, there is a bijective correspondence between maximal ideals and
homomorphisms onto $\mathbf 2$. In the last section of the paper we show that every multideal can be extended
to an ultramultideal, and that there exists  a bijective correspondence between ultramultideals and homomorphisms onto $\mathbf n$. Moreover, ultramultideals are proved to be exactly the prime multideals.


\section{Preliminaries\label{benzina}}
The notation and terminology in this paper are pretty standard. For
concepts, notations and results not covered hereafter, the reader is
referred to \cite{BS,MMT87} for universal algebra, to \cite{Lee,Leech,Spi} for skew BAs and to \cite{first,LPS13,BLPS18} for $n$BAs. 
Superscripts that mark the difference between operations and operation symbols will be dropped whenever the context is sufficient for a disambiguation. 

If $\tau$ is an algebraic type, an algebra $\mathbf{A}$ of type $\tau $ is
called \emph{a }$\tau $\emph{-algebra}, or simply an algebra when $\tau $ is
clear from the context. An algebra is \emph{trivial} if its carrier set is a singleton set.

$\mathrm{Con}(\mathbf{A})$ is the lattice of all
congruences on $\mathbf{A}$, whose bottom and top elements are,
respectively, $\Delta =\{(x,x):x\in A\}$ and $\nabla =A\times A$. 
Given $a,b\in A$, we write $\theta(a,b)$ for the smallest congruence $\theta$
such that $( a,b) \in \theta $.

We say that an algebra $\mathbf{A}$ is: (i) \emph{subdirectly irreducible}
if the lattice $\mathrm{Con}(\mathbf{A})$ has a unique atom; (ii) \emph{simple} if $\mathrm{Con}(\mathbf{A})=\{\Delta ,\nabla \}$; (iii) \emph{directly indecomposable} if 
$\mathbf{A}$ is not isomorphic to a direct product of two nontrivial algebras.

A class $\mathcal{V}$ of $\tau $-algebras is a \emph{variety} (equational class) if it is closed
under subalgebras, direct products and homomorphic images. 
If $K$ is a class of $\tau$-algebras, the variety $\mathcal V(K)$ generated by $K$ is the smallest variety including $K$. If $K=\{\mathbf A\}$ we write $\mathcal V(\mathbf A)$ for $\mathcal V(\{\mathbf A\})$.

\subsubsection{Notations}\label{notation}
If $A$ is a set and $X\subseteq A$, then $\bar X$ denotes the set $A\setminus X$.

Let $\hat n=\{1,\dots,n\}$ and $q$ be an operator of arity $n+1$.
If $d_1,\ldots,d_k$ is a partition of $\hat n$ and $x,y_1\ldots,y_k$ are  elements,
then
$$q(x, y_1/d_1,\ldots,y_k/d_k)$$
denotes  $q(x, z_1,\ldots,z_n)$ where for all $1\leq i\leq n$, $z_i=y_j$ iff
$i\in d_j$.

\noindent If $d_j$ is a singleton $\{i\}$, then we write $y/i$ for $y/d_j$.

\noindent If $d_i=\hat n\setminus d_r$ is a the complement of $d_r$, then we may  write $y/\bar d_r$ for $y/d_i$.




\subsection{Factor Congruences and Decomposition}
Directly indecomposable algebras play an important role in the characterisation of the structure of a variety of
algebras.
For example, if the class
of  indecomposable algebras in a Church variety (see Section \ref{dobbiaco} and \cite{first}) is universal, then any algebra in the variety is
a weak Boolean product of  directly indecomposable algebras.
In this section we summarize the basic ingredients of factorisation:
tuples of complementary factor congruences
and  decomposition operators (see  \cite{MMT87}) .

\begin{definition}
\label{thm:cong} A sequence $(\phi_{1},\dots,\phi_n)$ of congruences on a $\tau$-algebra $\mathbf{A}$ is a $n$-tuple of complementary factor congruences exactly when:
\begin{enumerate}
\item $\bigcap_{1\leq i\leq n}\phi_{i}=\Delta$;

\item $\forall (a_1,\dots,a_n)\in A^n$, there is $u\in A$ such that $a_i\phi_{i}\,u$,
for all $1\leq i\leq n$.
\end{enumerate}
\end{definition}

If $(\phi_{1},\dots,\phi_n)$  is a $n$-tuple of complementary factor congruences on $\mathbf{A}
$, then the function
$f:\mathbf{A}\rightarrow \prod\limits_{i=1}^n\mathbf{A}/\phi _{i}$,
defined by $f(a) =(a/\phi _{1},\dots,a/\phi _{n})$, is an
isomorphism. Moreover, every factorisation of $\mathbf A$ in $n$ factors univocally determines a $n$-tuple of complementary factor congruences.

A pair  $(\phi_{1},\phi_{2})$ of congruences is a pair of complementary factor
congruences if and only if $\phi_{1}\cap\phi_{2}=\Delta$ and $\phi_{1}\circ\phi_{2}=\nabla$. 
The pair $(\Delta,\nabla)$ corresponds to the product $\mathbf{A}\cong\mathbf{A}\times\mathbf{1}$, where $\mathbf{1}$ is a trivial algebra; obviously $\mathbf{1}\cong\mathbf{A}/\nabla$ and $\mathbf{A}\cong\mathbf{A}/\Delta$. 

A \emph{factor congruence} is any congruence which belongs to a pair of
complementary factor congruences. 
The set of factor congruences of $\mathbf A$ is not, in general, a sublattice of $\mathrm{Con}(\mathbf{A})$.

Notice that, if $(\phi_{1},\dots,\phi_n)$ is a $n$-tuple of complementary factor congruences, then $\phi_i$ is a factor congruence for each $1\leq i\leq n$, because
the pair $(\phi_i, \bigcap_{j\neq i} \phi_j)$ is a pair of complementary factor congruences. 

It is possible to characterise $n$-tuple of complementary factor congruences in terms of certain algebra
homomorphisms called \emph{decomposition operators} (see \cite[Def.~4.32]{MMT87} for additional details). 

\begin{definition}
\label{def:decomposition} 
An \emph{$n$-ary decomposition operator} on an algebra $\mathbf{A}$ is a function $f:A^{n}\rightarrow A$ satisfying the following identities:
conditions: 
\begin{description}
\item[D1] $f( x,x,\dots,x) =x$;
\item[D2] $f( f( x_{11},x_{12},\dots,x_{1n}),\dots,f(x_{n1},x_{n2},\dots,x_{nn}))=f(x_{11},\dots,x_{nn})$;
\item[D3] $f$ is an algebra homomorphism from $\mathbf{A}^n$ to $ \mathbf{A}$:\\
$f(g(x_{11},x_{12},\dots,x_{1k}),\dots, g(x_{n1},x_{n2},\dots,x_{nk})) = g(f(x_{11},\dots,x_{n1}),\dots, f(x_{1k},\dots,x_{nk}))$, for every $g\in \gt$ of arity $k$.
\end{description}
\end{definition}

There is a bijective correspondence between $n$-tuples of complementary factor
congruences and $n$-ary decomposition operators, and thus, between $n$-ary decomposition
operators and factorisations of an algebra in $n$ factors.

\begin{theorem}
\label{prop:pairfactor} Any $n$-ary decomposition operator $f:\mathbf{A}^{n}\rightarrow \mathbf{A}$ on an algebra $\mathbf{A}$ induces a $n$-tuple of
complementary factor congruences $\phi _{1},\dots,\phi _{n}$, where each $\phi _{i}\subseteq A\times A$ is defined by: 
\begin{equation*}
a\ \phi _{i}\ b\ \ \text{iff}\ \ f(a,\dots,a,b,a,\dots,a)=a\qquad(\text{$b$ at position $i$}).
\end{equation*}
Conversely, any $n$-tuple $\phi _{1},\dots,\phi _{n}$ of complementary factor
congruences induces a decomposition operator $f$ on $\mathbf{A}$: 
$f(a_1,\dots,a_n)=u$ iff $a_{i}\,\phi _{i}\,u$, for all $i$, where such an element $u$ is unique.
\end{theorem}

\subsection{Factor Elements}\label{sec:factor}
The notion of decomposition operator and  of factorisation can sometimes be internalised: some elements of the algebra,
the so called factor elements, 
can embody all the information codified by a decompostion operator.


Let $\mathbf A$ be a $\tau$-algebra, where we distinguish a $(n+1)$-ary term operation $q$.

\begin{definition}
  We say that an element $e$ of $\mathbf A$ is a \emph{factor element w.r.t. $q$} if the $n$-ary operation $f_e:A^n\to A$, defined by
$$f_e(a_1,\dots,a_n)= q^{\mathbf A}(e,a_1,\dots,a_n),\ \text{for all $a_i\in A$},$$
is a $n$-ary decomposition operator (that is, $f_e$ satisfies identities (D1)-(D3) of Definition \ref{def:decomposition}).
\end{definition}




An element $e$ of $\mathbf{A}$ is a factor element if, and only if, the tuple of relations $(\phi_1,\dots,\phi_n)$, defined by
$a\ \phi _{i}\ b$ iff $q(e,a,\dots,a,b,a,\dots,a)=a$ ($b$ at position $i$),
constitute a $n$-tuple of complementary factor congruences of $\mathbf{A}$.

By \cite[Proposition 3.4]{CS15} the set  of factor elements is closed under the operation $q$: if $a,b_1,\dots,b_n\in A$ are factor elements, then $q(a,b_1,\dots,b_n)$ is also a factor element.

We notice that 
\begin{itemize}
\item different factor elements may define the same tuple of complementary factor congruences.
\item there may exist $n$-tuples of complementary factor congruences that do not correspond to any factor element.
\end{itemize}
In Section \ref{trallallero} we describe a class of algebras, called Church algebras of dimension $n$, where the $(n+1)$-ary operator $q$ induces a bijective correspondence between a suitable subset of factor elements, the so-called $n$-central elements,  and the set of all $n$-ary decomposition operators. 
 
\subsection{Skew Boolean Algebras}\label{skewskew}
We  review here some basic definitions and results on  \emph{skew lattices} \cite{Lee89} and \emph{skew Boolean algebras} \cite{Leech}.

A \emph{skew lattice} is an algebra $\mathbf{A}=(A,\lor,\land)$ of type $(2,2)$,
where both $\lor$ and $\land$ are associative, idempotent binary operations, connected by the absorption law:
$x\lor (x\land y)= x= x\land (x\lor y)$; and  $(y\land x)\lor x=
x= (y\lor x)\land x$.
%
%
%
The absorption condition is equivalent to the
following pair of biconditionals:
$\text{$x\lor y= y$ iff $x\land y= x$; and $x\lor y= x$ iff $x\land y= y$%
}$.

In any skew lattice we define the following relations: 
\begin{enumerate}
\item $x\leq y$ iff $x\land y=x=y\land x$.
\item $x\preceq y$ iff $x\land y\land x=x$.
\item $x\preceq_{l}y$ iff $x\land y=x$.
\item $x\preceq_{r}y$ iff $y\land x=x$.
\end{enumerate}
The relation $\leq$ is a partial ordering, while the relations $\preceq, \preceq_{l}, \preceq_{r}$ are preorders. Observe that 
\begin{itemize}
\item[(i)]   The equivalence induced by $\preceq$, denoted as $\mathcal{D}$, is in
fact a congruence, and $\mathbf{A}/\mathcal{D}$ is the maximal lattice image of $\mathbf{A}$. 
\item[(ii)]  The equivalences $\mathcal{L}$ and $\mathcal{R}$, respectively induced by
$\preceq_l$ and $\preceq_r$, are again congruences; moreover, $\mathcal{L}$ is
the minimal congruence making $\mathbf{A}/\mathcal{L}$ a right-zero skew lattice (we recall that a skew lattice $\mathbf{A}$ is called \emph{right-zero} if  $a\land b=b$ for all $a,b\in A$). Similarly for $\mathbf{A}/\mathcal{R}$ and left-zero skew lattices.
\end{itemize}

A skew lattice is \emph{left-handed}  (\emph{right-handed}) if $\mathcal{L}=\mathcal{D}$ ($\mathcal{R}=\mathcal{D})$.
The following conditions are equivalent for a skew lattice $\mathbf{A}$:
(a) $\mathbf{A}$ is right-handed (left-handed); (b) for all $a,b\in A$, $a\land b\land a=b\land a$ ($a\land b\land a=a\land b$).



If we expand skew lattices by a subtraction operation and a constant $0$, we
get the following noncommutative variant of Boolean algebras (see \cite{Leech}).

\begin{definition}
A \emph{skew $\mathrm{BA}$} is an algebra $\mathbf{A}=(A,\lor ,\land ,\setminus,0)$ of
type $(2,2,2,0)$ such that:

\begin{itemize}
\item[(S1)] its reduct $(A,\lor ,\land)$ is a skew lattice satisfying the
identities 
\begin{itemize}
\item Normality: $x\land y\land z\land x= x\land z\land y\land x$;
\item Distributivity: $x\land (y\lor z)= (x\land y)\lor (x\land z)$ and $(y\lor z)\land x= (y\land x)\lor (z\land x)$;
\end{itemize}
\item[(S2)] $0$ is left and right absorbing w.r.t. multiplication;

\item[(S3)] the operation $\setminus $ satisfies the identities
\begin{itemize}
\item $(x\land y\land x)\lor (x\setminus y)= x=(x\setminus y)\lor (x\land y\land x)$; 
\item $x\land y\land x\land (x\setminus y)=0=(x\setminus y)\land x\land y\land x$.
\end{itemize}
\end{itemize}
\end{definition}

It can be seen that, for every $x\in A$,  the natural partial order of the subalgebra $x\land A \land x=\{x\land y\land x: y\in A\} = \{ y: y\leq x\}$ of $\mathbf A$ is a Boolean lattice. Indeed,  the algebra $(x\land A \land x, \lor,\land,0, x, \neg)$, where $\neg y=x\setminus y$ for every $y\leq x$, is a Boolean algebra with minimum $0$ and maximum $x$.

Notice that 
\begin{itemize}
\item The normal axiom implies the commutativity of $\land$ and $\lor$ in the interval  $x\land A \land x$.
\item Axiom (S2) expresses that $0$ is the minimum of the natural partial order on $A$.
\item Axiom (S3) implies that, for every $y\in x\land A \land x$, $x\setminus y$ is the complement of $y$ in the Boolean lattice $x\land A \land x$.
\end{itemize}

A nonempty subset $I$ of a skew BA $\mathbf A$ closed under $\lor$ is a \emph{$\preceq$-ideal} of $\mathbf A$ if it satifies one of the following equivalent conditions:
\begin{itemize}
\item $x\in A$, $y\in I$ and $x\preceq y$ imply $x \in I$;
\item $x\in A$ and $y\in I$ imply $x\land y, y\land x\in I$;
\item $x\in A$ and $y\in I$ imply $x\land y\land x\in I$.
\end{itemize} 
Given a congruence $\phi$ on a skew BA, $[0]_\phi$ is a  $\preceq$-ideal. Conversely,
every  $\preceq$-ideal $I$ is the $0$-class of a unique congruence $\phi$.

\subsection{A term equivalence result for skew $\mathrm{BA}$s} In \cite{CS15} Cvetko-Vah and the second author have introduced the variety of semicentral right Church algebras (SRCA) and have shown that the
variety of right-handed skew BAs is term equivalent to the variety of SRCAs.
It is worth noticing that, in SRCAs, a single ternary operator $q$ replaces all the binary operators of skew BAs.

 An algebra $\mathbf A=(A,q,0)$ of type $(3,0)$ is called a \emph{right Church algebra} (RCA, for short) if
 it satisfies the identity $q(0,x,y)= y$.

\begin{definition} Let $\mathbf A=(A,q,0)$ be a RCA. An element $x\in A$ is called \emph{semicentral} if it is a factor element (w.r.t. $q$) satisfying $q(x,x,0)=x$. 
\end{definition}

\begin{lemma} \cite[Proposition 3.9]{CS15} Let $\mathbf A=(A,q,0)$ be an $\mathrm{RCA}$. 
Every semicentral element $e\in A$ determines a pair of complementary factor congruences:
$$\text{$\phi_e= \{ (x,y) : q(e,x,y) = x\}$ and  $\bar\phi_e=\{ (x,y) : q(e,x,y) = y\}$}$$ 
such that $\phi_e= \theta(e,0)$, the least congruence of $\mathbf A$ equating $e$ and $0$.
\end{lemma}

\begin{definition}
 An algebra $\mathbf A=(A,q,0)$ of type $(3,0)$ is called a \emph{semicentral} $\mathrm{RCA}$ ($\mathrm{SRCA}$, for short) if every element of $A$ is semicentral.
 \end{definition}


\begin{theorem}\label{thm:equiv}
 The variety of right-handed skew $\mathrm{BA}$s is term equivalent to the variety of $\mathrm{SRCA}$s. 
\end{theorem}

The proof is based on the following correspondence between the algebraic similarity types
of skew BAs and of SRCAs:
\[%
\begin{array}
[c]{lll}%
q(x,y,z) & \rightsquigarrow & (x\land y) \lor (z\setminus x)\\
x\lor y & \rightsquigarrow & q(x,x,y)\\
x\land y & \rightsquigarrow & q(x,y,0)\\
y\setminus x & \rightsquigarrow & q(x,0,y).
\end{array}
\]

\begin{example}\label{partial-exa} (see \cite{CVLS13,CS15})
Let $\mathcal F(X,Y)$ be the set of all partial functions from $X$ into $Y$.  The algebra $\mathbf F= (\mathcal F(X,Y), q, 0)$ is a SRCA, where 
\begin{itemize}
\item $0=\emptyset$ is the empty function;
\item For all functions $f: F\to Y$, $g:G\to Y$ and $h:H\to Y$ ($F,G,H\subseteq X$),  
$$q(f,g,h) =g|_{G\cap F} \cup h|_{H\cap \bar F}.$$
\end{itemize} 
By Theorem \ref{thm:equiv} $\mathbf F$ is term equivalent to the right-handed skew BA of universe $\mathcal F(X,Y)$, whose operations are defined as follows:
$$f\land g = g|_{G\cap F};\qquad f\lor g = f\cup g|_{G\cap \bar F};\qquad g\setminus f = g|_{G\cap \bar F}.$$
\end{example}

\section{Boolean-like algebras of finite dimension\label{trallallero}}

Some important properties of Boolean algebras are shared  by $n$-pointed algebras whose elements satisfy all the equational conditions of $n$-central elements through an operator $q$ of arity $n+1$ satisfying the fundamental properties of a generalised if-then-else connective. These algebras, and the varieties they form, were termed Boolean-like algebras of dimension $n$ in \cite{BLPS18}.

\subsection{Church algebras of finite dimension\label{dobbiaco}}
In this section we recall from \cite{BLPS18} the notion of a Church algebra of dimension $n$. These algebras have $n$ designated elements $\e_1,\dots, \e_n$ ($n \geq 2$) and an operation $q$ of arity $n + 1$ (a sort of ``generalised if-then-else'') satisfying $q(\e_i, x^1,\dots, x^n) = x^i$. The operator $q$ induces, through the so-called $n$-central elements, a decomposition of the algebra into n factors. 

\begin{definition}
Algebras of type $\tau$ having $n$ designated elements $\e_1,\dots,\e_n$ ($n\geq 2$) and a term operation $q$ of arity $n+1$  satisfying $q(\e_i,x^1,\dots,x^n)=x^i$ are called \emph{Church algebras of dimension} $n$ ($n\mathrm{CA}$, for short); 
 $n$CAs admitting only the $(n+1)$-ary $q$ operator and the $n$ constants $\e_{1},\dots ,\e_{n}$ are called \emph{pure} $n$CAs.
\end{definition}

{If $\mathbf{A}$ is an }$n\mathrm{CA}${,} then $\mathbf{A}_{0}=(A,q,\e_{1},\dots ,\e_{n})$ is the \emph{pure reduct} of $\mathbf{A}$.

Church algebras of dimension $2$ were introduced as Church algebras in \cite{MS08} and studied in \cite{first}.  Examples of Church algebras of dimension $2$ are Boolean algebras (with $q(x,y,z) =(x\wedge y)\vee (\lnot x\wedge z)$) or rings with unit (with $q( x,y,z) =xy+z-xz$). Next, we present some examples of Church algebra having dimension greater than $2$.

\begin{example}
\label{cocchio} (\emph{Semimodules}) Let $R$ be a semiring and 
 $V$ be an $R$-semimodule generated by a finite set $E=\{\e_{1},\dots ,\e_n\}$. Then we define an operation $q$ of arity $n+1$ as
follows (for all $\mathbf{v}=\sum_{j=1}^nv_{j}\e_j$
and $\mathbf{w}^{i}=\sum_{j=1}^nw_{j}^{i}\e_j$):
$$
q(\mathbf{v},\mathbf{w}^{1},\dots ,\mathbf{w}^{n})=\sum_{i=1}^{n}v_{i}\mathbf{w}^{i}.
$$
Under this definition, $V$ becomes a $n\mathrm{CA}$. As a concrete example, if $B$ is a Boolean algebra, $B^n$ is a semimodule (over the Boolean ring $B$) with the following operations: $(a^1,\dots,a^n)+(b^1,\dots,b^n) = (a^1\lor b^1,\dots,a^n\lor b^n)$ and $b(a^1,\dots,a^n) = (b\land a^1,\dots,b\land a^n)$.  $B^n$ is also called a Boolean vector space (see \cite{GL,G14}).
\end{example}


\begin{example}
\label{exa:partition} (\emph{$n$-Sets}) Let $I$ be a set.
An \emph{$n$-subset} of $I$ is a sequence $(Y_{1},\dots ,Y_{n})$ of subsets 
$Y_{i}$ of $I$. We denote by $\mathrm{Set}_{n}(I)$ the family of all $n$-subsets of $I$.  $\mathrm{Set}_{n}(I)$ becomes a pure $n$CA if we define an $(n+1)$-ary operator $q$ and $n$ constants $\e_1,\dots,\e_n$ as follows, for all $n$-subsets $\mathbf{y}^{i}= (Y^i_{1},\dots,Y^i_{n})$:
$$
q( \mathbf{y}^0,\mathbf{y}^1,\dots ,\mathbf{y}^n) =(\bigcup\limits_{i=1}^{n}Y^0_{i}\cap Y_{1}^{i},\dots,\bigcup\limits_{i=1}^{n}Y^0_{i}\cap Y_{n}^{i});\quad \e_1=(I,\emptyset,\dots,\emptyset),\dots, \e_n=(\emptyset,\dots,\emptyset,I).
$$
\end{example}



{In \cite{vaggione}, Vaggione introduced the notion
of \emph{central element} to study algebras whose complementary factor
congruences can be replaced by certain elements of their universes. 
Central elements coincide with central idempotents in rings with
unit 
and with members of the centre in ortholattices. 
 


\begin{theorem} \cite{BLPS18}
\label{thm:centrale} If $\mathbf{A}$ is a $n\mathrm{CA}$ of type $\tau $
and $c\in A$, then the following conditions are equivalent:

\begin{enumerate}
\item $c$ is a factor element (w.r.t. $q$) satisfying the identity $q(c,\e_{1},\dots ,\e_{n})=c$;

\item the sequence of congruences $(\theta (c,\e_{1}),\dots
,\theta (c,\e_{n}))$ is a $n$-tuple of complementary factor congruences of $\mathbf{A}$;

\item for all $a^{1},\dots ,a^{n}\in A$, $q(c,a^{1},\dots ,a^{n})$ is the
unique element such that $a^{i}\ \theta (c,\e_{i})\ q(c,a^{1},\dots ,a^{n})$,
for all $1\leq i\leq n$;

%
%
\item The function $f_{c}$, defined by $f_{c}(x^{1},\dots
,x^{n})=q(c,x^{1},\dots ,x^{n})$, is a $n$-ary decomposition operator on $\mathbf{A}$ such that 
$f_{c}(\e_{1},\dots ,\e_{n})=c.$
\end{enumerate}
\end{theorem}

\begin{definition}
\label{def:ncentral} If $\mathbf{A}$ is a $n\mathrm{CA}$, then $c\in A$
is called \emph{$n$-central} if it satisfies one of the equivalent conditions of Theorem \ref{thm:centrale}.
A central element $c$ is \emph{nontrivial} if $c\notin\{\e_{1},\dots ,\e_{n}\}$.
\end{definition}

Every $n$-central element $c\in A$ induces a decomposition of $\mathbf{A}$ as a direct product of the 
algebras $\mathbf{A}/\theta (c,\e_{i})$, for $i\leq n$.


The set of all $n$-central elements of a $n\mathrm{CA}$ $\mathbf{A}$ is a subalgebra of the pure reduct of $\mathbf{A}$.
We denote by $\mathbf{Ce}_{n}(\mathbf{A})$ the algebra 
$(\mathrm{Ce}_{n}(\mathbf{A}),q,\e_{1},\dots ,\e_{n})$ 
of all $n$-central elements of an 
$n\mathrm{CA}$ $\mathbf{A}$.

\bigskip

Factorisations of arbitrary algebras in $n$ factors may be studied in terms of $n$-central elements of suitable $n$CAs of functions, as explained in the following example.

\begin{example} 
Let $\mathbf A$ be an arbitrary algebra of type $\tau$ and $F$ be a set of functions from $A^n$ into $A$, which includes the $n$ projections and all constant functions ($a_1,\dots,a_n\in A$):
\begin{enumerate}
\item $\e_i^\mathbf F(a_1,\dots,a_n)=a_i$;
\item $f_b(a_1,\dots,a_n)=b$, for every $b\in A$;
\end{enumerate}
and it is closed under the following operations (for all $f, h_i,g_j\in F$ and all $a_1,\dots,a_n\in A$):
\begin{enumerate}
\item[(3)] $q^\mathbf F(f,g_1\dots,g_n)(a_1,\dots,a_n) =f(g_1(a_1,\dots,a_n)\dots,g_n(a_1,\dots,a_n))$.
\item[(4)] $\sigma^\mathbf F(h_1,\dots,h_k)(a_1,\dots,a_n) = \sigma^\mathbf A(h_1(a_1,\dots,a_n),\dots, h_k(a_1,\dots,a_n))$, for every $\sigma\in\tau$ of arity $k$.
\end{enumerate}
The algebra $\mathbf F = (F,\sigma^\mathbf F,q^\mathbf F,\e_1^\mathbf F,\dots,\e_n^\mathbf F)_{\sigma\in\tau}$ is a $n$CA.
It is possible to prove that a function $f\in F$ is a $n$-central element of $\mathbf F$ if and only if 
 $f$ is a $n$-ary decomposition operator on the algebra $\mathbf A$ commuting with every function $g\in F$:
$$f(g(a_{11},\dots, a_{1n}),\dots, g(a_{n1},\dots, a_{nn})) = g(f(a_{11},\dots, a_{n1}),\dots, f(a_{1n},\dots, a_{nn})).$$
The reader may consult \cite{SLP18} for the case $n=2$.
\end{example}

\subsection{Boolean-like algebras}

Boolean algebras are Church algebras of dimension $2$ all of whose elements are 
$2$-central. It turns out that, among the $n$-dimensional Church algebras, those
algebras all of whose elements are $n$-central inherit many of the
remarkable properties that distinguish Boolean algebras. We now recall from \cite{BLPS18} the notion of
Boolean-like algebras of dimension $n$,  the main subject of study of this paper.

In \cite{BLPS18} $n$BAs are studied in the general case of an arbitrary similarity type.
Here, we restrict ourselves to consider the {\em pure} case, where $q$ is the unique operator of the
algebra.

\begin{definition}
\label{mezzucci}A pure $n\mathrm{CA}$ $\mathbf{A} = (A,q,\e_1,\dots,\e_n)$ is called a \emph{Boolean-like
algebra of dimension $n$} ($n\mathrm{BA}$, for short) if every element of $A$
is $n$-central.
\end{definition}

The class of all  $n\mathrm{BA}$s is a
variety axiomatised by the following identities:
\begin{itemize}
\item[$\qquad (\mathrm{B0})$]  $q(\e_i,x^1,\dots,x^n) = x^i$ ($i=1,\dots,n$).
\item[$\qquad (\mathrm{B1})$]    $q(y,x,\dots,x) = x$.
\item[$\qquad (\mathrm{B2})$] $q(y, q(y, x^{11},x^{12},\dots,x^{1n}),\dots,q(y,x^{n1},x^{n2},\dots,x^{nn}))= q(y,x^{11},\dots,x^{nn})$.
\item[$\qquad (\mathrm{B3})$]  $q(y,q(x^{10},x^{11},\dots,x^{1n}),\dots, q(x^{n0},x^{n1},\dots,x^{nn})) = q(q(y, x^{10},\dots,x^{n0}),\dots, q(y, x^{1n},\dots,x^{nn}))$.
\item[$\qquad (\mathrm{B4})$] $q(y,\e_1,\dots,\e_n) = y$.
\end{itemize}

\bigskip

Boolean-like algebras of dimension $2$ were introduced in \cite{first} with the
name \textquotedblleft Boolean-like algebras\textquotedblright . \emph{Inter
alia}, it was shown in that paper that the variety of Boolean-like algebras
of dimension $2$ is term-equivalent to the variety of Boolean algebras.

\begin{example}
The algebra $\mathbf{Ce}_{n}(\mathbf{A})$ of all $n$-central
elements of a $n\mathrm{CA}$ $\mathbf{A}$ of type $\tau$ is a canonical example of $n\mathrm{BA}$ (see the remark after Theorem \ref{thm:centrale}).
\end{example}

\begin{example}\label{exa:n}
The algebra 
$\mathbf{n}=( \{ \mathsf \e_{1},\dots,\mathsf \e_{n}\} ,q^{\mathbf{n}},\mathsf e^\mathbf{n}_{1},\dots,\mathsf e^\mathbf{n}_{n})$,
where $q^{\mathbf{n}}( \mathsf \e_{i},x^{1},\dots,x^{n}) =x^{i}$ for every $i\leq n$, is a $n\mathrm{BA}$.
\end{example}

\begin{example}\label{exa:parapa} (\emph{$n$-Partitions}) Let $I$ be a set. An \emph{$n$-partition} of $I$ is a $n$-subset $(Y^{1},\ldots ,Y^{n})$ of $I$ such that $\bigcup_{i=1}^{n}Y^{i}=I$ and $Y^{i}\cap Y^{j}=\emptyset $ for all $i\neq j$.
The set of $n$-partitions of $I$ is closed under the $q$-operator defined in 
Example \ref{exa:partition} and constitutes  the algebra of all $n$-central
elements of the pure $n$CA $\mathrm{Set}_{n}(I)$ of all $n$-subsets of $I$. 
Notice that the algebra  of $n$-partitions of $I$, denoted by $\mathrm{Par}_n(I)$, can be proved isomorphic to the $n\mathrm{BA}$ $\mathbf{n}^I$ (the Cartesian product of $I$ copies of the algebra $\mathbf{n}$).
\end{example}

The variety $\mathrm{BA}$ of Boolean algebras is semisimple as every $\mathbf{A}\in \mathrm{BA}$ is
subdirectly embeddable into a power of the $2$-element Boolean algebra, which
is the only subdirectly irreducible member of $\mathrm{BA}$. This property finds an analogue in the structure theory of $n\mathrm{BA}$s.

\begin{theorem}
\label{lem:subirr} \cite{BLPS18} The algebra $\mathbf n$ is the unique subdirectly irreducible $n\mathrm{BA}$ and it generates 
the variety of $n\mathrm{BA}$s.
\end{theorem}



The next corollary shows that, for any $n\geq 2$, the $n\mathrm{BA}$ $%
\mathbf{n}$ plays a role analogous to the Boolean algebra $\mathbf{2}$ of
truth values.

\begin{corollary}
\label{cor:stn} Every $n\mathrm{BA}$ $\mathbf{A}$ is isomorphic to a subdirect power of $\mathbf{n}^{I}$, for some set $I$.
\end{corollary}

%
%

A subalgebra of the $n$BA $\mathrm{Par}_n(I)$ of the $n$-partitions on a set $I$, defined in Example \ref{exa:parapa},  is called a \emph{field of $n$-partitions on $I$}. The Stone representation theorem for $n$BAs follows.

\begin{corollary} \label{cor:field-partitions}
Any $\mathrm{nBA}$ is isomorphic to a field of $n$-partitions on a suitable set $I$.
\end{corollary}

One of the most remarkable properties of the $2$-element Boolean algebra, called \emph{primality} in universal algebra \cite[ Sec. 7 in Chap. IV]{BS}, is the definability of all finite Boolean functions in terms of the connectives {\sc and, or, not}. This property is inherited by $n$BAs. An algebra of cardinality $n$ is
primal if and only if it admits the $n\mathrm{BA}$ $\mathbf n$ as subreduct.

\begin{definition} Let $\mathbf{A}$ be a nontrivial algebra.
$\mathbf{A}$ is \emph{primal} 
if it is of finite cardinality and, for every function $f:A^{k}\rightarrow A$ ($k\geq 0$), there is a $k$-ary term $t$ such that for all $a^{1},\dots,a^{k}\in A$, 
$f( a^{1},\dots,a^{k}) =t^{\mathbf{A}}(a^{1},\dots,a^{k})$.
\end{definition}

A variety $\mathcal V$ is primal if $\mathcal V= \mathcal V(\mathbf A)$ for a  primal algebra $\mathbf A$.


\begin{theorem}\label{prop:nbaprim} \cite{BLPS18}
\begin{itemize}
\item[(i)] The variety $n\mathrm{BA}=\mathcal{V}(\mathbf n)$ is primal;
\item[(ii)] Let $\mathbf{A}$ be a finite algebra of cardinality $n$. Then $\mathbf A$ is primal if and only if it admits  the algebra $\mathbf n$ as subreduct. 
\end{itemize}
\end{theorem}



\section{Skew Boolean algebras and $n$BAs}
In this section we prove that any $n\mathrm{BA}$ $\mathbf A$ contains a symmetric cluster of right-handed skew BAs $S_{1}(\mathbf A),\dots, S_{n}(\mathbf A)$.  The algebra $S_{ i}(\mathbf A)$, called the skew $i$-reduct of $\mathbf A$, has $\e_i$ as a bottom element,  and the other constants $\e_1,\dots,\e_{i-1},\e_{i+1},\dots,\e_n$ as maximal elements. 
Rather interestingly, every permutation $\gs$ of the symmetric group $S_n$ determines a bunch of isomorphisms
$$\mathrm{S}_1(\mathbf A) \cong \mathrm{S}_{\gs 1}(\mathbf A)\qquad\dots\qquad \mathrm{S}_n(\mathbf A) \cong \mathrm{S}_{\gs n}(\mathbf A)$$
which shows the inner symmetry of the $n$BAs.
We conclude the section with  a general representation theorem for right-handed skew BAs in terms of $n$BAs of $n$-partitions.

\subsection{The skew reducts of a $n$BA}
In \cite{CS15} it is shown that the variety of skew BAs is term equivalent to the variety of
SRCAs, whose type contains only a ternary operator.
Here we use the $n+1$-ary operator $q$ of a $n$BA $\mathbf A$  to define ternary operators $t_1,\ldots,t_n$ such that
the reducts $(A,t_i,\e_i)$ are isomorphic SRCAs. Their term equivalent skew BAs are all isomorphic reducts of $\mathbf A$, too.

For every subset $d$  of $\hat n$ and $i\in \hat n$, we denote by $\bar d$ the set $\hat n\setminus d$ and by $\bar i$ the set $\hat n\setminus\{i\}$.

The $(n+1)$-ary operator $q$ determines an operator $t_d$ of arity $3$, for every nonempty set $d\subseteq \hat n$.

\begin{definition}
Given $d\subseteq \hat n$, we define $t_d(x,y,z)= q(x, y/\bar d,z/ d)$.
%
%
\end{definition}


\begin{lemma}\label{lem:dec0} Let $\mathbf A$ be a $n\mathrm{BA}$. The following conditions hold for every $x,y\in A$:
\begin{enumerate}
\item $t_d(\e_i,x,y)=y$ and $t_d(\e_j,x,y)=x$, for every $i\in d$ and $j\notin d$; 
\item For every $x\in A$, $t_d(x,-,-)$ is a $2$-ary decomposition operator on $\mathbf A$.
\end{enumerate}
\end{lemma}

\begin{proof} (1) Trivial. (2) The binary operator $t_d(x,-,-)$ is a decomposition operator, because it is obtained by the $n$-ary decomposition operator $q(x,-,\dots,-)$  equating some of its coordinates  (see \cite{MMT87}).
%
\end{proof}


We introduce below two constants and five binary operations derived from $t_d$.

\begin{definition}
Let $0_i=\e_i$ and $1_j=\e_j$, where $i\in  d$ and $j\in\bar  d$. We define 
 $$x\land_d y=t_d(x,y,0_i);\qquad x\lor_d y= t_d(x,1_j,y);\qquad y\setminus_d x=t_d(x,0_i,y);$$ 
 $$x\: \barwedge_d\: y= t_d(x,y,x); \qquad x\ \overline\vee_d\ y= t_d(x,x,y).$$
\end{definition}

We now define three reducts that characterise the inner structure of a $n$BA.

\begin{definition}
  Let $\mathbf A$ be a $n\mathrm{BA}$, $d$ be a nonempty subset of $\hat n$,  $0_i=\e_i$ and $1_j=\e_j$ ($i\in  d$ and $j\in\bar  d$). We define the following three reducts of $\mathbf A$:
\begin{itemize}
\item[(i)] The Church $d$-reduct $C_{d}(\mathbf A)=(A, t_d, 0_i,1_j)$.
\item[(ii)] The right Church $d$-reduct $R_{d}(\mathbf A)=(A, t_d, 0_i)$.
\item[(iii)] The skew $d$-reduct $S_{d}(\mathbf A)=(A, \land_d, \overline\vee_d, \setminus_d, 0_i)$.
\end{itemize}
\end{definition}

\emph{Notation}: If $d=\{j\}$ is a singleton, we write $C_{j}(\mathbf A)$ for $C_{\{j\}}(\mathbf A)$. Similarly for the other reducts and for the operations. For example, we write $t_j$ for $t_{\{j\}}$.

The Church $d$-reduct of a $n$BA is a $2$CA of factor elements.  It is a $2$BA if and only if all its elements are meet idempotents.

\begin{proposition} Let $\mathbf A$ be a $n\mathrm{BA}$.
\begin{itemize}
\item[(i)]  Every element of the Church $d$-reduct $C_{d}(\mathbf A)$   is a factor element (w.r.t. $t_d$).
\item[(ii)] The map $c:A \to A$, defined by $$c(x)=t_d(x,1_j,0_i),$$ is an endomorphism of the Church $d$-reduct $C_{d}(\mathbf A)$, whose image is the $2\mathrm{BA}$ of its $2$-central elements.
\item[(iii)]  If $ d=\bar j$, then $x$ is a $2$-central element of $C_{\bar j}(\mathbf A)$ iff it is $\land_{\bar j}$-idempotent (i.e., $x\land_{\bar j} x = x$).
Therefore,  $C_{\bar j}(\mathbf A)$ is a $2\mathrm{BA}$ iff every element is $\land_{\bar j}$-idempotent.
\end{itemize}
\end{proposition}

\begin{proof} (i) By Lemma \ref{lem:dec0}(2).

(ii)  By (i) an element $x$ of the $2$CA $C_{d}(\mathbf A)$  is $2$-central iff $c(x)=x$.  The map $c$ is a homomorphism because 
$$c(t_d(x,y,z))=c(q(x,y/\bar d,z/d))=t_d(q(x,y/\bar d,z/d),1_j,0_i)=_{B3}$$ 
$$q(x,c(y)/\bar d,c(z)/d)=_{B3} q(c(x),c(y)/\bar d,c(z)/d)= t_d(c(x),c(y),c(z)).$$
The conclusion follows because $c(c(x))=t_d(t_d(x,1_j,0_i),1_j,0_i)=_{B3}c(x)$, for all $x$.

(iii) By \cite[Proposition 3.6]{CS15} a factor element $x$ satisfies the identity $t_{\bar j}(x,1_j,0_i)=x$ iff 
$x\land_{\bar j} x = x$ and $x\lor_{\bar j} x = x$. Then the conclusion follows if we prove that $x\lor_{\bar j} x = x$, for every $x\in A$:
$x\lor_{\bar j} x = t_{\bar j}(x,1_j,x) =_{B4} q(x,\e_j/ j,x/\bar j)=q(x,\e_j/j,q(x,\e_1,\dots,\e_n)/\bar j)=_{B2}q(x,\e_1,\dots,\e_n)=x$.
\end{proof}

Some properties of the derived binary operations are stated in the following lemma.


As a matter of notation, if $d=\{i_1,\dots,i_k\}$, then $\e_d/d$ means
$\e_{i_1}/i_1,\dots, \e_{i_k}/i_k$.

\begin{lemma}\label{lem:dec} Let $\mathbf A$ be a $n\mathrm{BA}$, $0_i=\e_i$ and $1_j=\e_j$ ($i\in  d$ and $j\in\bar  d$). The following conditions hold for every $x,y\in A$ and $d\subseteq \hat n$:
\begin{enumerate}
\item $x\:\barwedge_d\: y=q(x, y/\bar  d,\e_{d}/ d)$ and $x\:\overline\vee_d\: y=q(x,\e_{\bar  d}/\bar  d, y/ d)$;
\item $x\:\barwedge_d\: x=x$ and $x\:\overline\vee_d\: x = x$;
\item $\e_k\: \barwedge_d\: x=\e_k$ and $\e_k \land_d x=\e_i$, for every $k\in d$;
\item $\e_k\: \barwedge_d\: x= x = \e_k\land_d x$, for every $k\notin d$;
\item If $d=\bar j$, then $\overline\vee_{\bar j} = \lor_{\bar j}$;
\item If $d=\{ i\}$, then $\barwedge_{ i} = \land_{ i}$.
\end{enumerate}
\end{lemma}

\begin{proof} (1) $x\:\barwedge_d\: y= t_d(x,y,x) = q(x,y/\bar d,x/ d) =_{B4} q(x,y/\bar d,q(x,\e_1,\dots,\e_n)/ d)=_{B2} 
q(x,y/\bar d,\e_d/ d)$.

(2) $x\:\barwedge_d\: x= q(x,x/\bar d,x/ d)=q(x,x\dots,x)=_{B1}x$.

(5) $x\:\overline\vee_{\bar j}\: y=_{(1)}q(x,\e_{ j}/ j, y/ \bar j)=t_{\bar j}(x,1_j,y)=x\lor_{\bar j} y$.

(6) Similar to (5).
\end{proof}

%
%


We now characterise the right Church and the skew reducts of $\mathbf A$.


%
%
%
%
%
%
%
%

\begin{proposition}\label{red} Let $d=\{i\}$. The right Church $i$-reduct $R_{ i}(\mathbf A)=(A,t_i,0_i)$ of a $n\mathrm{BA}$ $\mathbf A$ is a $\mathrm{SRCA}$.
\end{proposition}

\begin{proof}
 By Lemma \ref{lem:dec0}(2) every element of $A$ is a factor element (w.r.t. $ t_{ i}$), and  by Lemma \ref{lem:dec}(2),(6) every element satisfies $x\land_i x=x\ \bar\land_i\ x= x$. Then every element of $A$ is semicentral and $R_{ i}(\mathbf A)$ is a $\mathrm{SRCA}$.
\end{proof}

Recall that if $d=\{ i\}$ then $\land_{ i} = \barwedge_{ i}$.

\begin{corollary} Let $d=\{i\}$. The skew $i$-reduct $S_{ i}(\mathbf A)=(A, \land_{ i}, \overline\vee_{ i}, \setminus_{ i}, 0_i)$ of $\mathbf A$ is a right-handed skew $\mathrm{BA}$ with bottom element $0_i=\e_i$ and maximal elements $\e_1,\dots,\e_{i-1},\e_{i+1},\dots,\e_n$.
\end{corollary}

\begin{proof} By Theorem \ref{thm:equiv}, by Lemma \ref{lem:dec}(6) and by Proposition \ref{red} $S_{ i}(\mathbf A)$ is a right-handed skew $\mathrm{BA}$. The natural partial order of  $S_{ i}(\mathbf A)$, defined in Section \ref{skewskew}, will be denoted by $\leq^i$.  The element $\e_k$ ($k\neq i$) is maximal, because, for every $x\neq\e_k$,
$\e_k \leq^i x$ implies $\e_k= \e_k \land_i x$,  contradicting  Lemma \ref{lem:dec}(4).
\end{proof}

Let $\mathbf A$ be a $n\mathrm{BA}$ and $x\in A$. The element $x$ is $n$-central on $\mathbf A$ and determines the $n$-tuple $(\theta^\mathbf A(x,\e_1),\dots,\theta^\mathbf A(x,\e_n))$ of complementary factor congruences on $\mathbf A$. Similarly, the element $x$ is semicentral on the SRCA $R_{ i}(\mathbf A)$ and determines the pair $(\phi^i_x,\bar\phi^i_x)$ of complementary factor congruences on  $R_{ i}(\mathbf A)$, where $\phi^i_x = \theta^{R_{ i}(\mathbf A)}(x,\e_i)$. The following proposition compares these factor congruences.

\begin{proposition} Let $\mathbf A$ be a $n\mathrm{BA}$ and $x\in A$. Then we have:
$$\phi^i_x= \theta^{R_{ i}(\mathbf A)}(x,\e_i)=\theta^\mathbf A(x,\e_i)=\{(a,b): t_i(x,a,b)=a\};\qquad\bar\phi^i_x = \bigcap_{j\neq i} \theta^\mathbf A(x,\e_j).$$
\end{proposition}

\subsection{A bunch of isomorphisms}

It turns out that all the skew reducts of a $n$BA $\mathbf A$ are isomorphic. In order to prove this,
we study the action of the symmetric group $S_n$ on  $\mathbf A$. 
The first part of this section is rather technical.

Let $\mathbf A$ be a $n$BA. 
For every permutation $\gs$ of the symmetric group $S_n$ and $x,y_1,\dots,y_n\in A$,
we define a sequence $u_s$ ($1\leq s\leq n+1$) parametrised by another permutation $\tau$:
$$u_{n+1}=y_{\tau n};\qquad u_s= t_{\tau s}(x,u_{s+1},y_{\gs \tau s})\quad (s\leq n).$$
In the following lemma we prove that $u_1$ is independent of the permutation $\tau$.

Notice that $u_n=q(x,y_{\tau n}/\overline{\tau n},y_{\gs \tau s}/\tau n)$ and  $u_s=q(x,u_{s+1}/\overline{\tau s},y_{\gs \tau s}/\tau s)$.

\begin{lemma}\label{lem:strano} We have:
 $$u_s=q(x,y_{\tau n}/\overline{\{\tau1,\tau2\dots,\tau(s-1)\}},y_{\gs\tau s}/\tau s,y_{\gs\tau(s+1)}/\tau(s+1),\dots,y_{\gs\tau n}/\tau n).$$
 Then $u_1 = q(x,y_{\gs\tau 1}/\tau 1,y_{\gs\tau 2}/\tau 2,\dots,y_{\gs\tau n}/\tau n) = q(x,y_{\gs 1},y_{\gs 2},\dots,y_{\gs n})$.

\end{lemma}

\begin{proof}
Assume that 
$$u_{s+1}=q(x,y_{\tau n}/\overline{\{\tau1,\tau2\dots,\tau s\}},y_{\gs\tau(s+1)}/\tau(s+1),\dots,y_{\gs\tau n}/\tau n).$$ Then we have:
\[
\begin{array}{lll}
u_s  & =  &  t_{\tau s}(x,u_{s+1},y_{\gs \tau s})  \\
  & =  & q(x,u_{s+1}/\overline{\tau s}, y_{\gs\tau s}/\tau s)  \\
  &  =_{B2} &   q(x,y_{\tau n}/\overline{\{\tau1,\tau2\dots,\tau s-1\}},y_{\gs\tau s}/\tau s,y_{\gs\tau(s+1)}/\tau(s+1),\dots,y_{\gs\tau n}/\tau n).
\end{array}
\]
\end{proof}


%
%

We define
$$x^\sigma= q(x,\e_{\gs 1},\e_{\gs 2},\dots,\e_{\gs n}).$$

The permutation $(ij)$ exchanges $i$ and $j$: $(ij)(i)=j$ and $(ij)(j)=i$.

\begin{lemma}\label{sigma} The following conditions hold, for all permutations $\sigma$ and $\tau$: 
\begin{enumerate}
\item   $t_{i}(x, t_{j}(x,y,z),u)= t_{j}(x, t_{i}(x,y,u),z)=q(x, y/\overline{\{i,j\}},z/j,u/i)$ ($i\neq j$);
 \item $q(x, y_{\gs 1},\dots, y_{\gs n})=t_{\tau 1}(x, t_{\tau 2}(x, t_{\tau 3}(x,(\dots t_{\tau n}(x,y_{\tau n},y_{\gs \tau n})),y_{\gs \tau 3})),y_{\gs \tau 2}),y_{\gs \tau 1})$;


\item $x^{\gs}= t_{\tau 1}(x, t_{\tau 2}(x, t_{\tau 3}(x,(\dots t_{\tau n}(x,\e_{\tau n},\e_{\gs \tau n}),\e_{\gs \tau 3})),\e_{\gs \tau 2}),\e_{\gs \tau 1})$;
\item $q(x^\gs, y_1,\dots, y_n)=q(x, y_{\gs 1},\dots, y_{\gs n})$;
\item $x^{(ij)} =  t_{i}(x, t_{j}(x,x,\e_i),\e_j) =  t_{j}(x, t_{i}(x,x,\e_j),\e_i)$;
\item $x^{\gt\circ\gs}=(x^{\gs})^{\gt}$;
\item $q(x, y_1,\dots, y_n)^\gs = q(x,(y_1)^{\gs},\dots, (y_n)^{\gs})$.
\end{enumerate}
 
\end{lemma}

\begin{proof} 
(1)

$$\begin{array}{lll} t_{i}(x, t_{j}(x,y,z),u) &=& q(x,   t_{j}(x,y,z)/\bar i,u/i)\\
&=& q(x, q(x, y/\bar j ,z/j)/\bar i,u/i)\\
&=_{B2}& q(x, y/\overline{\{i,j\}},z/j,u/i)\\
&=_{B2}&q(x, q(x, y/\bar i ,u/i)/\bar j,z/j) \\
&=& q(x,   t_{i}(x,y,u)/\bar j,z/j)\\
&=& t_{j}(x, t_{i}(x,y,u),z).
\end{array}$$

(2) is the unfolding of the definition of $u_1$.

(3) follows from (2) by putting $y_i=\e_i$.

(4) $q(x^\gs, y_1,\dots, y_n)=q(q(x,\e_{\gs 1},\dots,\e_{\gs n}), y_1,\dots, y_n)=_{B3}q(x, y_{\gs 1},\dots, y_{\gs n})$.

(5) $$\begin{array}{llll}
t_{i}(x, t_{j}(x,x,\e_i),\e_j)&=&t_{j}(x, t_{i}(x,x,\e_j),\e_i)&\text{by (1)}\\
&=&q(x, x/\overline{\{i,j\}},\e_i/j,\e_j/i)&\text{by (1)}\\
&=&q(x, q(x,\e_1,\dots,\e_n)/\overline{\{i,j\}},\e_i/j,\e_j/i)&\text{by (B4)}\\
&=&x^{(ij)}&\text{by (B2)}
\end{array}$$


(7) $$\begin{array}{lll} q(x, y_1,\dots, y_n)^\gs &=& q(q(x, y_1,\dots, y_n),\e_{\gs 1},\e_{\gs 2},\dots,\e_{\gs n})\\
&=_{B3}&
q(x,q(y_1,\e_{\gs 1},\e_{\gs 2},\dots,\e_{\gs n}),\dots,q(y_n,\e_{\gs 1},\e_{\gs 2},\dots,\e_{\gs n}))\\ 
&=& q(x,y^{\gs}_1,\dots, y_n^{\gs}). 
\end{array}$$
\end{proof}

\begin{theorem} For every transposition $(rk)\in S_n$, the map  $x\mapsto x^{(rk)}$ defines an isomorphism from $\mathrm{S}_{r}(\mathbf A)$ onto  $\mathrm{S}_k(\mathbf A)$.
\end{theorem}

\begin{proof} The map $x\mapsto x^{(rk)}$ is bijective because $(x^{(rk)})^{(rk)} = x$ by Lemma \ref{sigma}(6) and (B4). The map is a homomorphism:
 $t_{r}(x,y,z)^{(rk)}=_{L.\ref{sigma}(7)} t_{r}(x,y^{(rk)},z^{(rk)}) =
t_{r}((x^{(rk)})^{(rk)},y^{(rk)},z^{(rk)}) = q((x^{(rk)})^{(rk)},y^{(rk)}/\overline{r},z^{(rk)}/r)=_{L.\ref{sigma}(4)} q(x^{(rk)},y^{(rk)}/\overline{k},z^{(rk)}/k)= t_{k}(x^{(rk)},y^{(rk)},z^{(rk)})$. Moreover, $(\e_{r})^{(rk)}=\e_{(rk)r}=\e_k$.
\end{proof}


\subsection{A general representation theorem for right-handed skew BA}
In this section we show that, for every $n\geq 3$, there is a representation of an arbitrary right-handed skew BA within
a suitable $n$BA of $n$-partitions (described in Example \ref{exa:parapa}).

\begin{theorem} Let $n\geq 3$. Then
 every right-handed skew $\mathrm{BA}$ can be embedded into the skew $ i$-reduct $S_{ i}(\mathbf A)$  of a suitable $n\mathrm{BA}$ $\mathbf A$ of $n$-partitions.
\end{theorem}

\begin{proof} 
 
(a) By \cite[Corollary 1.14]{Leech} 
every right-handed skew $\mathrm{BA}$ can be embedded into an algebra of partial functions with codomain the set $\{1, 2\}$ (see Example \ref{partial-exa}), where  $0=\emptyset$ is the empty function, 
$f\land g = g|_{G\cap F}$, $f\lor g = f\cup g|_{G\cap \overline F}$ and  $g\setminus f = g|_{G\cap \overline F}$ (with $F,G$ and $H$ the domains of the functions $f,g,h$, respectively).

(b) By Corollary \ref{cor:field-partitions} every $n$BA is isomorphic to a $n$BA of $n$-partitions of a suitable set $I$ (see Examples \ref{exa:partition} and \ref{exa:parapa}). 
If $P=(P_1,\dots,,P_n)$ and $Q=(Q_1,\dots,Q_n)$ are $n$-partitions of $I$, then 
%
\begin{equation}\label{ee}P\land_{ i} Q =  t_{ i}(P,Q,\e_i)=q( P,Q/\bar i, \e_i/i)=
(\overline{P_i}\cap Q_1,\dots, P_i\cup (\overline{P_i}\cap Q_i),
\dots,\overline{P_i}\cap Q_n).\end{equation}
The other operations can be similarly defined.

(c) We define an injective function $^*$ between the set of partial functions from a set $I$ into  $\{1, 2\}$ and the set of $n$-partitions of $I$. If
$f:I\rightharpoonup \{1, 2\}$ is a partial function, then $f^*=(P_1,\dots,P_n)$ is the following $n$-partition of $I$: $P_1 = f^{-1}(1)$, $P_2 = f^{-1}(2)$, $P_i=A\setminus \mathrm{dom}(f)$ and $P_k=\emptyset$ for any $k\neq 1,2,i$.

(d) The map $^*$ preserves the meet. Let $f:F\to  \{1, 2\}$ and $g:G \to  \{1, 2\}$ ($F,G\subseteq I$) be functions. Then by (1) we derive $(f\land g)^* = f^*\land_{ i} g^*$ as follows:
$$f^*=(f^{-1}(1), f^{-1}(2),\emptyset,\dots,\emptyset,\overline F,\emptyset,\dots,\emptyset);\quad
g^*=(g^{-1}(1), g^{-1}(2),\emptyset,\dots,\emptyset,\overline G,\emptyset,\dots,\emptyset)$$
\[
\begin{array}{lll}
 (f\land g)^* & =  & (g|_{G\cap F})^*  \\
  & =  & (F\cap g^{-1}(1), F\cap g^{-1}(2),\emptyset,\dots,\emptyset,\overline G\cup \overline F,\emptyset,\dots,\emptyset)  \\
  & =  &   (F\cap g^{-1}(1),F\cap g^{-1}(2),\emptyset\dots, \overline F\cup (F\cap \overline G),\emptyset,\dots,\emptyset)\\
  & =  &f^*\land_{ i} g^*.
\end{array}
\]
Similarly for the other operations.
%
%
%
%
%
%
 \end{proof}

\section{Skew star algebras}
The skew reducts of a $n$BA are so deeply related that they allow to recover the full structure of the $n$BA.
It is worthy to introduce a new variety of algebra, called {\em skew star} algebras, equationally axiomatising 
$n$ isomorphic skew BAs and their relationships. In the main result of this section we prove that the variety of 
skew star algebras is term equivalent to the variety of $n$BAs.

%

\begin{definition}
 An algebra $\mathbf B=(B,t_i,0_i)_{1\leq i\leq n}$, where $t_i$ is ternary and $0_i$ is a constant, is called
 a \emph{skew star algebra} if the following conditions hold, for every $j\neq i$:
 \begin{itemize}
\item[(N0)] $(B,t_i,0_i)$ is a SRCA,  for every $1\leq i\leq n$.
\item[(N1)] $t_i(0_j,y,z)= y$.   
\item[(N2)] $t_{ 1}(x, t_{ 2}(x, t_{ 3}(x,(\dots t_{n-1}(x,0_{ n},0_{n-1})\dots),0_3),0_{  2}),0_{  1})= x$.
\item[(N3)]  $t_{i}(x, t_{j}(x,y,z),u)= t_{j}(x, t_{i}(x,y,u),z)$.
\item[(N4)] $t_i(x,y,z)= t_1(x,t_2(x,t_3(x,\dots t_{i-1}(x,t_{i}(x,t_{i+1}(x,\dots,y),z),y)\dots,y),y),y)$.
\item[(N5)] $t_i(x,-,-)$ is a homomorphism of the algebra $(A,t_j,0_j)\times (A,t_j,0_j)$ into $(A,t_j,0_j)$: 
$$t_i(x,t_j(y^1,y^2,y^3),t_j(z^1,z^2,z^3)) = t_j(t_i(x,y^1,z^1),t_i(x,y^2,z^2),t_i(x,y^3,z^3)).$$
\end{itemize}
\end{definition}

Skew star algebras constitute a variety of algebras.

As usual, we define the operations
$x\land_i y=t_i(x,y,0_i);\quad x\ \bar\lor_i\ y=t_i(x,x,y);\quad y\setminus_i x=t_i(x,0_i,y)$.

\begin{proposition} Let $\mathbf B$ be a skew star algebra. Then the following conditions hold for $j\neq i$:
 \begin{itemize}
\item[(i)] $t_i(0_i,x,y)=y$ is equivalent to $$0_i\land_i y = 0_i = y \land_i 0_i;\quad 0_i\ \bar\lor_i\ y = y= y\ \bar\lor_i\ 0_i ;\quad y\setminus_i 0_i= y;\quad 0_i\setminus_i y= 0_i.$$
\item[(ii)] (N1) is equivalent to $0_j\land_i y = y;\quad 0_j\ \bar\lor_i\ y = 0_j;\quad y\setminus_i 0_j= 0_i$.
\item[(iii)] (N2) is equivalent to $x\land_1 (x\land_2(\dots (x\land_{n-1} 0_n)\dots)) = x$.
\item[(iv)] (N3) implies the following identities: 
$x\land_i (x\land_j y)= x\land_j (x\land_i y);\quad x\land_i(x\ \bar\lor_j\ y)= x\ \bar\lor_j\ y.$
\end{itemize}
\end{proposition}

\begin{proof} Since $(B,t_i,0_i)$ is a SRCA, then by Theorem \ref{thm:equiv} we have 
$t_i(x,y,z)= (x\land_i y)\ \bar \lor_i\ (z\setminus_i x)$.

(i) ($\Leftarrow$) $t_i(0_i,y,z)= (0_i\land_i y)\ \bar \lor_i\ (z\setminus_i 0_i) =0_i\ \bar \lor_i\ z=z$.

(ii)  ($\Leftarrow$) $t_i(0_j,y,z)= (0_j\land_i y)\ \bar \lor_i\ (z\setminus_i 0_j)=y\ \bar\lor_i\
0_i=_{(i)}y$.

(iii) By definition of $\land_i$.

(iv) Trivial.
\end{proof}

Consider the following correspondence between the algebraic similarity types of $n$BAs and of skew star algebras.
\begin{itemize}
\item Beginning on the $n$BA side: $t_i(x,y,z):=q(x,y/\bar i,z/i)$ and $0_i:=\e_i$. 
\item Beginning on the skew star algebra side: 
$$q_t(x,y_1,\dots,y_n):= t_1(x,t_2(x,t_3(x,t_4(\dots t_{n-1}(x,y_n,y_{n-1})\dots),y_3),y_2),y_1);\quad \e_i:=0_i.$$
\end{itemize}
If $\mathbf B$ is a skew star algebra, then $\mathbf B^\bullet = (B; q_t, \e_1,\dots,\e_n)$ denotes the corresponding algebra in the similarity type of $n$BAs. Similarly, if $\mathbf A$ is a $n$BA, then $\mathbf A^* = (A;t_1,\dots,t_n,0_1,\dots,0_n)$ denotes the corresponding algebra in the similarity type of skew star algebras.

It is not difficult to prove the following theorem.

\begin{theorem}
  The above correspondences define a term equivalence between the varieties of $n\mathrm{BA}$s and of skew star algebras. More precisely,
  \begin{itemize}
\item[(i)] If $\mathbf A$ is a $n\mathrm{BA}$, then $\mathbf A^*$ is a skew star algebra;
\item[(ii)] If $\mathbf B$ is a skew star algebra, then $\mathbf B^\bullet$ is a $n\mathrm{BA}$;
\item[(iii)] $(\mathbf A^*)^\bullet = \mathbf A$;
\item[(iv)] $(\mathbf B^\bullet)^* = \mathbf B$.
\end{itemize}
\end{theorem}

\begin{proof} (i) (N0) and (N1) derive from Proposition \ref{red} and Lemma \ref{lem:dec0}.
 (N2) follows from Lemma \ref{sigma}(3), by putting $\tau$ and $\sigma$ to be the identical permutation. 
(N3) is a consequence of Lemma \ref{sigma}(1).
For (N4) 
 we apply Lemma \ref{sigma}(1) to
$t_1(x,t_2(x,t_3(x,\dots t_{i-1}(x,t_{i}(x,t_{i+1}(x,\dots,y),z),y)\dots,y),y),y)$. We get the conclusion:
$$t_i(x,t_2(x,t_3(x,\dots t_{i-1}(x,t_{i+1}(x,\dots,y),y)\dots,y),y),z)=_{B2}t_i(x,y,z).$$
(N5) follows from (B3). 
 
(ii) (B0) derives from (N0) and (N1). 
By (N0) and (N5), $t_i(x,-,-)$ ($1\leq i\leq n$) is a decomposition operator on $\mathbf B$.
Then, for every $x\in B$, $q_t(x,-,\dots,-)$ is a $n$-ary decomposition operator on $\mathbf B^\bullet$ (i.e., (B1)-(B3) hold), because decomposition operators are closed under composition (see \cite{MMT87}). 
(B4) is a consequence of (N2).

(iii) Let $\mathbf A$ be a $n$BA. Since $t_i(x,y,z)=q(x,y/\bar i,z/i)$, then we have $q_t=q$ by  Lemma \ref{sigma}(2) with $\sigma$ and $\tau$ the identical permutation. 

(iv) Let $\mathbf B$ be a skew star algebra. By (N4) we have  that $t_i(x,y,z)=q_t(x,y/\bar i,z/i)$.
\end{proof}

\section{Multideals}
The notion of ideal plays an important role
in order theory and universal algebra. Ideals, filters
and congruences are interdefinable in Boolean algebras: 
$x\in I$ if and only if $\neg x\in \neg I$ if and only if
$x \theta_I 0$  if and only if
$\neg x \theta_I 1$, for every ideal $I$.
In the case of $n$BAs, the couple $(I,\neg I)$ is replaced by a $n$-tuple
$(I_1,\ldots,I_n)$ satisfying some compatibility conditions that extend in a conservative way those of the Boolean case.

\begin{definition}\label{def:fide}
Let $\mathbf A$ be a $n$BA. A \emph{multideal} is a $n$-partition $(I_1,\dots, I_n)$ of a subset $I$ of $A$ such that
\begin{enumerate}
\item[(m1)] $\e_k\in I_k$;
\item[(m2)] $a\in I_r$, $b\in I_k$ and $y_1,\dots,y_n\in A$ imply $q(a,y_1,\dots,y_{r-1},b,y_{r+1},\dots,y_n)\in I_k$;
\item[(m3)]  $a\in A$ and $y_1,\dots,y_n\in I_k$ imply $q(a,y_1,\dots, y_n)\in I_k$.
\end{enumerate}
The set $I$ is called the {\em carrier} of the multideal.
A \emph{ultramultideal} of $\mathbf A$ is a multideal whose carrier is $A$.
\end{definition}


The following Lemma, whose proof is straightforward, shows the appropriateness of the notion of multideal. In Section \ref{sec:multvscong} we show that there exists a bijective correspondence between multideals and congruences. 

\begin{lemma}\label{lem:congr}
If $\theta$ is  a proper congruence on a $n\mathrm{BA}$ $\mathbf A$, then $I(\theta)=(\e_1/\theta,\ldots,\e_n/\theta)$ is a
multideal of $\mathbf A$.
\end{lemma}

Multideals extend to the $n$-ary case the fundamental notions of (boolean) ideal and filter, as shown in the following Proposition.

Recall that a   $2$BA $\mathbf A = (A,q,e_1,e_2)$ is term equivalent to the Boolean
algebra $\mathbf {A^*}=(A,\wedge, \vee,\neg,0,1)$ where $0=e_2$, $1=e_1$,
$x\wedge y= q(x,y,0)$, $x\vee y= q(x,1,y)$, $\neg x=q(x,0,1)$ \cite{first}.

\begin{proposition}\label{prop:cons_ext}
 Let $\mathbf A$ be a  $2\mathrm{BA}$, and $I_1,I_2\subseteq A$. Then $(I_1, I_2)$ is a multideal of $\mathbf A$
  if and only if $I_2$ is an ideal of  $\mathbf {A^*}$, and
  $I_1=\neg I_2$ is the   filter associated to $I_2$ in  $\mathbf {A^*}$.
\end{proposition}
\begin{proof}
   If  $(I_1, I_2)$ is a multideal, then $0=e_2\in I_2$.
  Moreover if $x,y\in I_2$ then
  $x\vee y=q(x,1,y)\in I_2$  by (m2),
  and if  $x\in I_2$  and $y\in A$, then 
  $y\wedge x= q(y,x,0)\in I_2$ by (m3).
  The fact that $I_1=\{\neg x\ |\ x\in I_2\}=\{q(x,0,1)\ | x\in I_2\}$ follows from Lemma \ref{lem:idperm} below.
  Conversely, if $I_2$ is a Boolean ideal of $\mathbf {A^*}$ and $I_1=\neg I_2$,
  then the condition (m1) is clearly satisfied. Concerning (m2),
it is worth noticing that $q(x,z,y)=(x\wedge z)\vee(\neg x\wedge y)$. Then
if $x\in I_2, y\in I_1,z \in A$ (for instance, the other 3 cases being
similar to this one), we have that $\neg x\in I_1$, so that $\neg x\wedge y\in I_1$ and we conclude
that $(x\wedge z)\vee(\neg x\wedge y)=q(x,z,y)\in I_1$.
Concerning (m3), if $x,y\in I_2$ and $z\in A$ then $z \wedge x,\neg z\wedge y\in I_2$, hence  $(z \wedge x)\vee (\neg z\wedge y)=q(z,x,y)\in I_2$.
If $x,y\in I_1$ and $z\in A$, then $(z\wedge x)\vee(\neg z \wedge y)\geq
 (z\wedge (x\wedge y))\vee(\neg z \wedge (x\wedge y))=(z\ \vee \neg z)\wedge (x\wedge y)=x\wedge y\in I_1$, so that $(z\wedge x)\vee(\neg z \wedge y)=q(z,x,y)\in I_1$.
\end{proof}

In the $n$-ary case, multideals  
of $\mathbf A$
may be characterised as $n$-tuples of skew ideals in the  skew star algebra
associated to  $\mathbf A$, satisfying the conditions expressed in the
following Proposition.

\begin{proposition}\label{prop:caract_multideals}
  Let $\mathbf A$ be a $n\mathrm{BA}$, 
  $(I_1,\dots, I_n)$ be a  $n$-partition of a subset $I$ of $A$, and 
  $\mathbf A^* = (A;t_1,\dots,t_n,0_1,\dots,0_n)$ the skew star algebra 
  corresponding to $\mathbf A$. Then  $(I_1,\dots, I_n)$ is a multideal if and only if
  the following conditions are satisfied:

  \begin{enumerate}
\item[(I1)] $0_r\in I_r$;
\item[(I2)] $a\in I_r$, $b\in I_k$ and $y\in A$ imply
$t_r(a,y,b)\in I_k$.
\item[(I3)] $a,b\in I_r$  and $y\in A$  imply $t_k(y,a,b)\in I_r$, for all $k$.
\end{enumerate}
\end{proposition}
\begin{proof}
  Showing that a multideal satisfies I1, I2 and I3 is straightforward.
  A $n$-partition satisfying  I1, I2 and I3, trivially  verifies (m1).
  Concerning (m2), let us suppose that $a\in I_r$, $b\in I_k$ 
  and $y_1,\ldots,y_n\in A$. In order to show that
  $q(a,y_1,\dots,y_{r-1},b,y_{r+1},\dots,y_n) \in I_k$, we apply
Lemma \ref{sigma}(4):

$q(x, y_{\gs 1},\dots, y_{\gs n})=t_{\tau 1}(x, t_{\tau 2}(x, t_{\tau 3}(x,(\dots t_{\tau n}(x,y_{\tau n},y_{\gs \tau n})),y_{\gs \tau 3})),y_{\gs \tau 2}),y_{\gs \tau 1})$

in the case $\sigma=id, \tau=(1r)$, and we get
 \begin{align*}
 q(a,y_1,\dots,y_{r-1},b,y_{r+1},\dots,y_n)  & = & \\
  t_r(a,t_2(a,\ldots,y_2),b)\in I_k, \text{ by I2.}
  \end{align*}

  Concerning (m3), let $a_1,\ldots,a_n\in I_k$ and $y\in A$. We have

 $q(y,a_1,\dots,a_n)  =
 t_1(y,t_2(y,t_3(\ldots t_{n-2}(y,t_{n-1}(y,a_n,a_{n-1}),a_{n-2})\ldots,a_3),a_2),a_1)$

 By applying I3 $n$ times, we conclude that  $q(y,a_1,\dots,a_n) \in I_k$, since
 $t_{n-1}(y,a_n,a_{n-1})\in I_k$, hence  $t_{n-2}(y,t_{n-1}(y,a_n,a_{n-1}),a_{n-2})\in I_k$, and so on.
  \end{proof}

By using the characterisation of Proposition \ref{prop:caract_multideals} it is easy to see
that the components of a multideals are skew ideals in the skew Boolean algebra corresponding to their index.

Recall from Section \ref{skewskew} that $x \preceq^i y$ iff $x\land_iy\land_i x=x$.
\begin{corollary}
  If $(I_1,\dots, I_n)$ is a multideal of a $n\mathrm{BA}$ $\mathbf A$ 
  and $1\leq i\leq n$, then $I_i$ is  a $\preceq^i$-ideal of the skew $i$-reduct
   $S_i(\mathbf A)=(A, \land_i, \overline\vee_i, \setminus_i, 0_i)$.
\end{corollary}
\begin{proof}
Since  $S_i(\mathbf A)$ is right-handed, a non empty set $K\subseteq A$ is an ideal of $S_i(\mathbf A)$
  if and only if, for all $a,b\in K$ and $x\in A$, $a \overline\vee_i b\in K$ and
  $a \land_i x\in K$ (see Section \ref{skewskew}). Given $a,b\in I_i$ and  $x\in A$, we have 
  $a \overline\vee_i b = t_i(a,a,b)\in I_i$ and  $a \land_i x = t_i(a,x, 0_i)\in I_i$,
  by using in both cases the condition (I2) of Proposition \ref{prop:caract_multideals}
  (notice that $0_i\in I_i$, by (I1)).
 \end{proof}

\begin{lemma} The carrier $I$ of a multideal $(I_1,\dots, I_n)$ of a $n\mathrm{BA}$ $\mathbf A$  is a subalgebra of  $\mathbf A$.
\end{lemma}
 \begin{proof}
The constants $\e_1,\ldots,\e_n$ belong to $I$ by (m1).

\noindent If $a\in I_r$ and $b\in I_k$, then $q(a,y_1,\dots,y_{r-1},b,y_{r+1},\dots,y_n)\in I_k$, for all
 $y_1,\dots,y_n\in A$, by (m2). Hence $I$ is a subalgebra of $\mathbf A$.
\end{proof}

Any component $I_i$ of a multideal $(I_1,\dots, I_n)$ determines the multideal completely, as
shown in the following Lemma.

\begin{lemma}\label{lem:idperm} If $(I_1,\dots, I_n)$ is a multideal of a $n\mathrm{BA}$ $\mathbf A$, then $I_k = I_r^{(rk)}$  for all $r,k$.
\end{lemma}
\begin{proof} Let $a\in I_r$. Then $a^{(rk)} = t_r(a,t_k(a,a,\e_r),\e_k)\in I_k$ by
  Lemma \ref{sigma}(5) and
  Proposition \ref{prop:caract_multideals}(I2). Then we have
$$I_r^{(rk)}\subseteq I_k;\qquad I_k^{(rk)}\subseteq I_r.$$
The conclusion follows because $(a^{(rk)})^{(rk)} = a$, by Lemma \ref{sigma}(6) and (B4).
\end{proof}

Multideals are closed under arbitrary nonempty componentwise intersection.
The minimum multideal is the sequence $(\{\e_k\})_{k\in \hat n}$.
Given a  $n$BA $\mathbf A$, and $A_1,\dots,A_n\subseteq A$,
let us consider the set $\mathcal A$ 
of multideals containing $(A_1,\dots,A_n)$.
The {\em ideal closure} of $(A_1,\dots,A_n)$ is the componentwise intersection
of the elements of $\mathcal A$, if ${\mathcal A}\not=\emptyset$. Otherwise,
the ideal closure of $(A_1,\dots,A_n)$
is the constant $n$-tuple $I^\top=(A,\ldots,A)$, that we consider as a degenerate multideal,
by a small abuse of terminology.

As a matter of fact, $I^\top$ is the only degenerate multideal.

\begin{lemma}\label{lem:top}
  Let $\mathbf A$ be a $n\mathrm{BA}$ and  $I=(I_1,\dots,I_n)$ be a tuple of subsets of $A$ satisfying the closure properties
  of Definition \ref{def:fide}. The following are equivalent:
\begin{itemize}
\item[(i)] there exist $x\in A$, $ r\neq k$ such that
$x\in I_r\cap I_k$.
\item [(ii)] there exist $r\neq k$  such that $\e_k\in I_r$.
\item [(iii)] $I=I^\top$.
\end{itemize}
\end{lemma}

\begin{proof}
(i)$\Rightarrow$(ii): since $x\in I_k$, by Lemma \ref{lem:idperm}
we have that $x^{(rk)}\in I_r$. By Definition \ref{def:fide}(m2), we conclude 
that $q(x, \e_k/\overline{r}, x^{(rk)}/r)=_{B2}q(x,\e_k,\ldots,\e_k)=\e_k\in I_r$.

\noindent  (ii)$\Rightarrow$(iii): given  $y\in A$, we have 
$y=q(\e_k,y/\overline{r},\e_r/r)\in I_r $ by Definition \ref{def:fide}(m2).
Hence $I_r=A$ and the result follows from Lemma  \ref{lem:idperm} since 
$A^{(rk)}=A$ for all $1\leq k\leq n$.

\noindent  (iii)$\Rightarrow$(i): trivial.
\end{proof}

%
%
%
%

\section{The relationship between multideals and congruences}\label{sec:multvscong}
In Lemma \ref{lem:congr} we have seen that, for any congruence $\theta$ on a $n$BA, the equivalence classes
$\e_i/\theta$ form a multideal, exactly as in the Boolean case $0/\theta$ is an ideal
and $1/\theta$ the corresponding filter. Conversely,
in the Boolean case, any ideal $I$ (resp. filter $F$)
defines the congruence $x\theta_Iy \Leftrightarrow x\oplus y\in I$ (resp.
$x\theta_Fy \Leftrightarrow x \leftrightarrow y\in F$).
Rephrasing this latter correspondence in the
$n$-ary case is a bit more complicated.

\subsection{The Boolean algebra of coordinates} 
 Let $\mathbf A$ be a $n\mathrm{BA}$, $a\in A$ and $i\in \hat n$.
We consider the factor congruence $\theta^i_a = \theta(a,\e_i) = \{(x,y): t_i(a,x,y)=x\}$ generated by $a$.

Recall that  $\preccurlyeq_r^i$ and $\leq^i$, denote the
right preorder and the partial order of the skew Boolean algebra $\mathbf{S}_i(A)=(A,\land_i,\bar\lor_i,\setminus_i,\e_i)$ respectively (see Section \ref{skewskew}). 

\begin{lemma}
  \begin{itemize}
    \item  $(\e_1/\theta^i_a,\dots,\e_n/\theta^i_a)$ is a multideal of  $\mathbf A$.
    \item $\e_i/\theta^i_a=\{x\in A: x \preccurlyeq_r^i a\}$.
      \end{itemize}
\end{lemma}

\begin{proof}
  The first item follows from Lemma \ref{lem:congr}. For the second one:
  by definition of $\theta^i_a$, we have $x \theta^i_a y$ iff $t_i(a,x,y)=x$.
  Then if $x\in\   \e_i/\theta^i_a$, we have 
  $x= t_i(a,x,\e_i) =a \land_i x$, therefore by definition of $\preccurlyeq_r^i $
  we conclude $x\preccurlyeq_r^i a$.
\end{proof}

The following proposition is a consequence of \cite{CS15}, Proposition 4.15, by observing that $t_i(a,x,y)=x$ for every $x,y\in \e_i/\theta^i_a$.

\begin{proposition}\label{prop:cc}
\begin{itemize}
\item[(i)] The set $\e_i/\theta^i_a$ is a subalgebra of the right Church $i$-reduct
  $(A,t_i,\e_i)$.
\item[(ii)] The algebra 
  $(\e_i/\theta^i_a,t_i,\e_i, a)$  is a $2\mathrm{CA}$.
  \item [(iii)]The set  $\downarrow_i\!\! a=\{x : x \leq^i a\}$ is the Boolean algebra of 2-central elements of $(\e_i/\theta^i_a,t_i,\e_i, a)$.

\end{itemize}
\end{proposition}

 Notice that $a$ is maximal because, if $a\leq^i x\in \e_i/\theta^i_a$, then $a=a\land_i x =t_i(a,x,\e_i)=x$.
 
 \bigskip


We now specialise the above construction to the case $a=\e_j$ for a given $j\neq i$.

\begin{definition} Let $\mathbf A$ be a $n\mathrm{BA}$ and $i\neq j$. The \emph{Boolean center of $\mathbf A$}, denoted by $B_{ij}$, is the Boolean algebra of 2-central elements of the $2$CA $(A,t_i,\e_i, \e_j)$.
\end{definition}
By Proposition \ref{prop:cc} the carrier set of $B_{ij}$ is the set $\downarrow_i\!\!\e_j=\{x\in A: x\leq^i \e_j\}$
and we call {\em Boolean} any element of $B_{ij}$.

\begin{remark}
  The Boolean algebra $B_{ij}$ was defined in \cite{BLPS18} in a different but equivalent way (see  \cite[Section 6.1, Lemma 7(iii)]{BLPS18}).
  \end{remark}

Recall that the Boolean operations on $B_{ij}$ are defined as follows: 
 $$x\land_{ij} y= t_i(x,y,\e_i);\qquad x\lor_{ij} y= t_i(x,\e_j, y);\qquad \neg_{ij}(x)=t_i(x,\e_i,\e_j).$$
Remark that $\land_{ij}=\land_i$ and $\lor_{ij}=\bar\lor_i$ since
$x\bar\lor_iy= t_i(x,x,y)=t_i(x,t_i(x,\e_j,\e_i),y)=_{B2} t_i(x,\e_j,y)=x\lor_{ij} y$.
In the following we use the notation $\land_i, \bar\lor_i, \neg_{ij}$ for denoting the Boolean operations of $B_{ij}$.

In  \cite{BLPS18} a representation theorem is proved,
showing that any given $n$BA 
$\mathbf A$
can be embedded into  the $n$BA
of the $n$-central elements of the Boolean vector space 
$B_{ij}\times\ldots \times B_{ij}=B_{ij}^n$ (see Example \ref{cocchio}).
The proof of this result makes an essential use of the notion of
{\em coordinates} of elements of $\mathbf A$, that are
$n$-tuples of elements of $B_{ij}^n$, codifying the elements of
$\mathbf A$ as ``linear combinations'' of the constants (see Lemma \ref{lem:piu}(3)). In this paper, the notion of coordinate is again a central one,
being used to define the congruence associated to a multideal. In order
to highlight their relationship with the skew reducts of $\mathbf A$,
here we define the coordinates in terms of the $t_k$ operations.



\begin{definition}\label{def:coor}
 The \emph{coordinates} of $x\in A$ are the elements $x_k=t_k(x,\e_i,\e_j)$, for $1\leq k\leq n$.
\end{definition}

Notice that $x_i=\neg_{ij}(x)$. 


\begin{lemma}\label{lem:meet} Let $x,y^1,\ldots,y^n\in A$. We have:
\begin{itemize}
\item[(i)]  $x_k\land_i x_r = \e_i$ for all $k\neq r$.
\item[(ii)]  $x_1\bar\lor_i\dots\bar\lor_i x_n=\e_j$.
  \item[(iii)] 
$q(x,y^{1},\dots,y^{n})_k=q(x,(y^{1})_k,\dots,(y^{n})_k)=(x_1\land_i {(y^{1})}_k)\bar\lor_i\dots\bar\lor_i
(x_n\land_i {(y^{n})}_k)$.
\item[(iv)] $(x\land_i y)_k = x\land_i y_k$, for every $k\neq i$.
\item[(v)] $x_k\land_i x= x_k\land_i \e_k$, for every $k\neq i$.
\item[(vi)] $x_i\land_i x=\e_i$.
\item[(vii)]   If $x\in B_{ij}$, then
$$x_k=\begin{cases}\neg_{ij}x&\text{if $k=i$}\\ x&\text{if $k=j$}\\ \e_i&\text{otherwise}\end{cases}$$

\end{itemize}
\end{lemma}
 
\begin{proof} 
(i)-(vi) It is sufficient to check in the generator $\mathbf n$ of the variety $n$BA. 

%
%

(vii)  ($k=i$): By definition of $\neg_{ij}$ we have $x_i = t_i(x,\e_i,\e_j)=\neg_{ij}(x)$.

($k\neq i,j$): $x_k = t_k(x,\e_i,\e_j) = t_k(x\land_i \e_j,\e_i,\e_j)=t_k(t_i(x, \e_j,\e_i),\e_i,\e_j) = 
t_i(x,\e_i,\e_i)=\e_i$.

($k=j$): $x_j = t_j(x,\e_i,\e_j) = t_j(x\land_i \e_j,\e_i,\e_j)=t_j(t_i(x, \e_j,\e_i),\e_i,\e_j) = 
t_i(x,\e_j,\e_i)=x$.
\end{proof}

\begin{proposition}\label{lem:bool}
 The following conditions are equivalent for an element $x\in A$:
 \begin{itemize}
 \item[(a)] $x$ is Boolean;
  \item[(b)]  $x\land_i \e_j=x$; 
\item[(c)]   $x=y_k$, for some $y\in A$ and index $1\leq k\leq n$;
\item[(d)]  $x=x_j$;
\item[(e)]  $x_k=\e_i$, for every $k\neq i,j$;
\item[(f)]  $x=(x_i)_i$.
\end{itemize}
\end{proposition}

\begin{proof} (a) $\Leftrightarrow$ (b):  We have that $x\leq^i \e_j$ iff $x\land_i \e_j = x$ and $\e_j\land_i x = x$.
 The conclusion is obtained because the latter equality is  trivially true.

(c) $\Rightarrow$ (b):  $y_k\land_i \e_j=t_i(t_k(y,\e_i,\e_j),\e_j,\e_i)=_{B3} q(y,\e_j/k,\e_i/\bar k)= t_k(y,\e_i,\e_j)=y_k$.

(b) $\Rightarrow$ (d): If $x\land_i \e_j=x$, then $x_j=t_j(x,\e_i,\e_j) = t_j(x\land_i \e_j,\e_i,\e_j) =
t_j(t_i(x,\e_j,\e_i),\e_i,\e_j)= t_i(x,\e_j,\e_i)=x\land_i \e_j=x$.

(d) $\Rightarrow$ (c): Trivial.

(a) $\Rightarrow$ (e): By Lemma \ref{lem:meet}(iv),(vii) $x_k=(x\land_i \e_j)_k = x\land_i (\e_j)_k=\e_i$, for every $k\neq i$.

(e) $\Rightarrow$ (d): By Lemma \ref{lem:meet}(ii) the join of all coordinates of $x$ in $B_{ij}$ is the top element $\e_j$. By hypothesis (e) we derive $x_i\bar\lor x_j=\e_j$. Then, by applying the distributive property of $\land_i$ w.r.t. $\bar\lor_i$ of skew BAs, we obtain:
$x=\e_j\land_i x=_{L.\ref{lem:meet}(ii)}(x_i\bar\lor_i x_j)\land_i x=(x_i\land_i x) \bar\lor_i(x_j \land_i x)=_{L.\ref{lem:meet}(vi)} \e_i \bar\lor_i(x_j \land_i x) = x_j \land_i x=_{L.\ref{lem:meet}(v)} x_j \land_i \e_j=x_j$.

(b) $\Rightarrow$ (e): Let $k\neq i,j$. Then we have: $x_k=(x\land_i\e_j)_k=t_i(x,\e_j,\e_i)_k=t_i(x,(\e_j)_k,(\e_i)_k)=t_i(x,\e_i,\e_i)=\e_i$.

%
%
%

(f) $\Leftrightarrow$ (b): $(x_i)_i = t_i(t_i(x,\e_i,\e_j),\e_i,\e_j)=t_i(x, \e_j,\e_i)=x\land_i \e_j$. Then $(x_i)_i =x$ iff $x\land_i \e_j=x$.
 \end{proof}
 
 By Lemma \ref{lem:meet}(iv) and Lemma \ref{lem:bool}(d)  $x\land_i y$ is a Boolean element, for every $x\in A$ and $y\in B_{ij}$.

\subsection{The congruence defined by a multideal}

Let $\mathbf A$ be a $n$BA and $B_{ij}$ be the Boolean center of $\mathbf A$.

\begin{lemma} Let $I$ be a multideal on $\mathbf A$. Then $I_*=B_{ij} \cap I_i$ is a Boolean ideal and $I^*=B_{ij} \cap I_j$ is the Boolean filter complement of $I_*$.
\end{lemma}

\begin{proof} Recall that, in $B_{ij}$, $\e_i$ is the bottom element, $\e_j$ is the top element and $x\in B_{ij}$ iff $x\land_i \e_j=x$. 
We prove that $I_*$ is a Boolean ideal. First $\e_i\in I_*$.
If $x,y\in I_*$ and $z\in B_{ij}$, then we prove $x\bar\lor_i y, x\land_i z\in I_*$. 
By Proposition \ref{prop:caract_multideals}(I2) $x\bar\lor_i y$ and $x\land_i z$ belong to $I_i$. Moreover, we have:
$$(x\bar\lor_i y)\land_i \e_j=t_i(t_i(x,x,y),\e_j,\e_i)=_{B3} t_i(x,t_i(x,\e_j,\e_i),t_i(y, \e_j,\e_i))=t_i(x,x,y)=x\bar\lor_i y.$$ 
$$x\land_i z\land_i \e_j=t_i(t_i(x,z,\e_i),\e_j,\e_i)= t_i(x,t_i(z,\e_j,\e_i),\e_i)=t_i(x,z,\e_i)=x\land_i z,$$ 
because $x,y,z$ are Boolean.

We now show that $I^*$ is the Boolean filter complement of $I_*$. 

($x\in I_*\Rightarrow \neg_{ij}x\in I^*$): As $x\in I_i\cap B_{ij}$, then by Prop. \ref{prop:caract_multideals}(I2) $\neg_{ij}x=t_i(x,\e_i,\e_j)\in I_j\cap B_{ij}$. 

($\neg_{ij}x\in I^*\Rightarrow x\in I_*$): 
As $t_i(x,\e_i,\e_j)\in I_j$, then $x=\neg_{ij}\neg_{ij}x= t_i(t_i(x,\e_i,\e_j),\e_i,\e_j)\in I_i$.
%
%
\end{proof}


The following lemma characterises multideals in terms of coordinates.

\begin{lemma}\label{lem:boh}   Let $x\in A$. 
  \begin{itemize}
   \item [(a)]  $x\in I_r$ if and only if $x_r\in I^*$.
  \item[(b)] If $x\in I_i$, then $x_k \in I_*$ for every  $k\neq i$.
   \end{itemize}
\end{lemma}

\begin{proof}
(a) We start with $r=i$.\\
  ($\Rightarrow$) It follows from $x_i = t_i(x,\e_i,\e_j)$, because $x\in I_i$ and $e_j\in B_{ij}\cap I_j$.

 \noindent  ($\Leftarrow$) By hypothesis $x_i\in B_{ij}\cap I_j$. By Lemma \ref{lem:idperm} we have that $x_i^{(ij)}\in I_i$. Now $x_i^{(ij)} =q(x_i, \e_{(ij)1},\dots,\e_{(ij)n})= q(t_i(x,\e_i,\e_j), \e_{(ij)1},\dots,\e_{(ij)n})=t_i(x,\e_j,\e_i)=x\land_i \e_j \in I_i$. Now $x\in I_i$ if we are able to prove that $x\preceq^i   x\land_i \e_j$, because $I_i$ is a $\preceq^i$-ideal of the skew $i$-reduct of $\mathbf A$. Recalling that $x\preceq^i   x\land_i \e_j\Leftrightarrow x\land_i \e_j\land_i x = x$, we conclude: $x\land_i \e_j\land_i x= x\land_i x=x$ because $\e_j\land_i x=x$.
 

 We analyse $r\neq i$.
 By definition of $x^{(ir)}$ we derive $$(x^{(ir)})_i =  t_i(q(x,\e_{(ir)1},\dots,\e_{(ir)n}),\e_i,\e_j) =_{B3} q(x,\e_j/r,\e_i/\bar r)= t_r(x,\e_i,\e_j)=x_r.$$ 
Then,

$x_r\in I^*\Leftrightarrow  x_r=(x^{(ir)})_i\in I^*\Leftrightarrow x^{(ir)}\in I_i \Leftrightarrow_{L.\ref{lem:idperm}} x= (x^{(ir)})^{(ir)}\in I_r$.

(b)  By   $x_k = t_k(x,\e_i,\e_j)$ and $x\in I_i$.
%
%
\end{proof}


We consider the homomorphism $f_I: B_{ij}\to B_{ij}/I_*$ and
we define on $A$ the following equivalence relation:
$$x\theta_I y \Leftrightarrow \forall k.f_I(x_k)=f_I(y_k),$$
 where $x_k,y_k$ are the $k$-coordinates of $x$ and $y$, respectively (see Definition \ref{def:coor}).

\begin{proposition} $\theta_I$ is a congruence on $\mathbf A$.
\end{proposition}

\begin{proof} Let $a,b,x^1,y^1,\dots,x^n,y^n$ be elements such that $a \theta_I b$ and $x^k\theta_I y^k$, for every $k$. Then $q(a,x^1,\dots,x^n)\theta_Iq(b,y^1,\dots,y^n)$
iff $\forall k. f_I(q(a,x^1,\dots,x^n)_k)=f_I(q(b,y^1,\dots,y^n)_k)$. The conclusion follows because $f_I$ is a Boolean homomorphism and by Lemma \ref{lem:meet}(iii)
 $q(a,x^1,\dots,x^n)_k= q(a,(x^1)_k,\dots,(x^n)_k)=(a_1\land_i {(x^{1})}_k)\bar\lor_i\dots\bar\lor_i
(a_n\land_i {(x^{n})}_k)$.
\end{proof}

We define a new term operation to be used in Theorem \ref{thm:main}:
$$x+_iy=q(x,t_i(y,\e_i,\e_1),\dots,y,\dots,t_i(y,\e_i,\e_n));\quad \text{($y$ at position $i$).}$$ 


\begin{lemma}\label{lem:piu} Let $x,y\in A$ and $x_1,\dots,x_n$ be the coordinates of $x$. Then
\begin{enumerate}
\item $x+_i\e_i=\e_i+_i x=x$;
\item $x+_iy=y+_ix$;
\item $x+_i \e_k=\e_k+_ix = x_i\land_i\e_k$ ($k\neq i$).
\item $x+_ix=\e_i$;



\item The value of the expression $E\equiv (x_1\land_i\e_1)+_i((x_2\land_i \e_2)+_i(\dots+_i (x_n\land_i\e_n))\dots)$
is independent of the order of its parentheses. Without loss of generality,  we write 
$(x_1\land_i\e_1)+_i(x_2\land_i \e_2)+_i\dots+_i(x_n\land_i\e_n)$ for the expression $E$. Then we have:
 $$(x_1\land_i\e_{1})+_i(x_2\land_i \e_2)+_i\dots+_i (x_n\land_i\e_{n})=x.$$
\item If $x$ and $y$ have the same coordinates, then $x=y$.

\end{enumerate}
\end{lemma}

\begin{proof}  It is easy to check identities (1)-(5) in the generator $\mathbf n$ of the variety $n$BA. (6) is a consequence of (5).
\end{proof}

\begin{theorem}\label{thm:main}
  Let $\phi$ be a congruence and $H=(H_1,\dots,H_n)$ be a multideal of $\mathbf A$. Then
  $$\theta_{I(\phi)}=\phi\  \text{ and } I(\theta_H)=H.$$
\end{theorem}

%
%
%
%
\begin{proof} We first prove $I(\theta_H)_k=H_k$, for all $k$.  
  Recall that $\e_j$ is the $i$-coordinate of $\e_i$ and $\e_i$ is the $k$-coordinate of $\e_i$ for every $k\neq i$. 

 (1)   First we provide the proof for $k=i$. Let $x\in I(\theta_H)_i$.
If $x\theta_H \e_i$ then $f_H(x_i)=f_H((\e_i)_i)=f_H(\e_j)$, that implies $x_i\in H^*$. By Lemma  \ref{lem:boh}(a) we get the conclusion $x\in H_i$.

For the converse, let $x\in H_i$.  By Lemma  \ref{lem:boh}(a)  we have $x_i\in H^*$ and by Lemma \ref{lem:boh}(b) $x_k\in H_*$ for all $k\neq i$.
This implies $f_H(x_i)=f_H(\e_j)=f_H((\e_i)_i)$ and $f_H(x_k)=f_H(\e_i)=f_H((\e_i)_k)$ for all $k\neq i$, that implies $x\theta_H \e_i$. Since  $I(\theta_H)_i=\e_i/\theta_H$, we conclude.

(2) Let now $k\neq i$. By Lemma \ref{lem:idperm} we have $H_k = H_i^{(ik)}$. Let $x\in H_k$. Then $x=y^{(ik)}$ for some $y\in H_i$.
As, by (1), $y\theta_H \e_i$, then we have $x=y^{(ik)}\theta_H (\e_i)^{(ik)}=\e_k$.
 Since  $I(\theta_H)_k=\e_k/\theta_H$, we conclude.
 Now, assuming $x\theta_H\e_k$, we have: $y = (x)^{(ik)}\theta_H (\e_k)^{(ik)} = \e_i$. Then $y\in H_i$ and $x=y^{(ik)}\in H_k$.

\bigskip

%

  Let $\phi$ be a congruence.

(a) Let $x\phi y$. Then $\forall h.\ x_h\phi y_h$. Since $\phi$ restricted to $B_{ij}$ is also a Boolean congruence, then we obtain $(x_h\oplus_{ij} y_h)\phi \e_i$, where $\oplus_{ij}$ denotes the symmetric difference in the Boolean center $B_{ij}$.
   We now prove that $x\theta_{I(\phi)}y$. We have 
  $x\theta_{I(\phi)} y$ iff $\forall h.\ f_{I(\phi)}(x_h)=f_{I(\phi)}(y_h)$ iff $\forall h.\ x_h\oplus_{ij} y_h\in I(\phi)_*=
  B_{ij}\cap \e_i/\phi$ iff $\forall h.\ x_h\oplus_{ij} y_h\in \e_i/\phi$ 
  iff $\forall h.\ (x_h\oplus_{ij} y_h) \phi \e_i$. This last relation is proved above and we conclude $x\theta_{I(\phi)}y$.

(b) Let $x\theta_{I(\phi)}y$. Then $\forall h.\ x_h\oplus_{ij} y_h\in \e_i/\phi$ that implies  $\forall h.\ x_h\phi y_h$, because $\phi$ restricted to $B_{ij}$ is a Boolean congruence.
  Since by Lemma \ref{lem:piu}(5) there is a $n$-ary term $u$ such that $x=u(x_1,\dots,x_n)$ and $y=u(y_1,\dots,y_n)$,  then we conclude $x\phi y$ by using $\forall h.\ x_h\phi y_h$.
\end{proof}

\subsection{Ultramultideals}
In the Boolean case, there is a bijective correspondence between maximal ideals and
homomorphisms onto $\mathbf 2$. In this section we show that every multideal can be extended
to an ultramultideal, and that there exists  a bijective correspondence between ultramultideals and homomorphisms onto $\mathbf n$. We also show that prime multideals coincide with ultramultideals.

Let $(I_1,\dots,I_n)$ be a multideal of a $n\mathrm{BA}$ $\mathbf A$ and  $U$ be a Boolean ultrafilter of $B_{ij}$ that extends $I^{*}=B_{ij}\cap I_j$, and so $\bar U = B_{ij}\setminus U$ extends $I_{*}=B_{ij}\cap I_i$.

\begin{lemma} \label{lem:fond}
For all $x\in A$,  there exists a unique $k$ such that $x_k\in U$.
\end{lemma}
\begin{proof}
By Lemma \ref{lem:meet}(ii) the meet of two distinct coordinates is the bottom element $\e_i$. Then at most one coordinate may belong to $U$. On the other hand, if all coordinates belong to $\bar U$, then the top element $\e_j$ belong to $\bar U$. 
\end{proof}
Let $(G_k)_{k\in \hat n}$ be the sequence such that $G_k=\{x\in A: x_k\in U\}$, which, by Lemma \ref{lem:fond}, is well defined.

\begin{lemma}\label{lem:decisivo} $(G_{k})_{k\in \hat n}$ is an ultramultideal which extends $(I_{k})_{k\in \hat n}$.
\end{lemma}
\begin{proof}
  (m1) $\e_k\in G_k$ because  $(\e_k)_k=\e_j\in U$.

  \noindent (m2): let $x\in G_r$, $y\in G_{k}$, and $z^1,\dots, z^n\in A$.
By Lemma \ref{lem:meet}(ii),  $$q(x,z^1,\ldots,z^{r-1},y,z^{r+1}\ldots,z^n)_{k} = [\overline{\bigvee}_{s\not=r}(x_s\land_i (z^s)_k)]\ \bar\lor_i\ (x_r\land_i y_k).$$
Since $x_r,y_k\in U$, then $x_r\land_i y_k\in U$, and so $x_r\land_i y_k\sqsubseteq [\overline{\bigvee}_{s\not=i}(x_s\land_i (z^s)_k)]\ \bar\lor_i\ (x_r\land_i y_k)\in U$. Hence, $q(x,z^1,\ldots,z^{r-1},y,z^{r+1}\ldots,z^n)\in G_k$.

\noindent (m3) can be proved similarly.

\noindent We now prove that $(G_{k})_{k\in \hat n}$ extends $(I_{k})_{k\in \hat n}$. It is sufficient
to show that, for every $x\in I_k$, we have that $x_k\in U$. We get the conclusion by Lemma \ref{lem:boh}(i).
\end{proof}

\begin{theorem}
  \begin{itemize}
  \item [(i)] Every multideal can be estended to an ultramultideal.
  \item  [(ii)] There is a bijective correspondence between ulramultideals and homomorphisms onto $\mathbf n$.
    \end{itemize}
    \end{theorem}
    \begin{proof}
(i) follows from lemma \ref{lem:decisivo}. Regarding (ii), we remark that the algebra $\mathbf n$ is the unique simple $n$BA.
      \end{proof}

We conclude this section by characterising prime multideals.

\begin{definition}
We say that  a multideal $(I_1,\dots,I_n)$ is \emph{prime} if $x\land_i y \in I_i$ implies $x\in I_i$ or $y\in I_i$.
\end{definition}

\begin{proposition} 
 A multideal is prime iff it is an ultramultideal.
\end{proposition}

\begin{proof} ($\Rightarrow$) Let $(I_1,\dots,I_n)$ be a prime ideal. If $x\in B_{ij}$, then $x\land_i\neg_{ij}(x)=\e_i\in I_i$. Then either $x$ or $\neg_{ij}(x)\in I_i$. This implies that $I_*=B_{ij}\cap I_i$ is a maximal Boolean ideal and the complement $I^*=B_{ij}\cap I_j$ is a Boolean ultrafilter. 

\noindent Let now $y\in A$ such that $y\notin I=\bigcup_{k=1}^n I_k$.  By Lemma \ref{lem:boh}(a) we have that 
$y\in I_r$ iff $y_r\in I^*$. 
Then $y_r\notin I^*$ for all $r$. Since $I^*$ is a Boolean ultrafilter, then $y_r\in I_*$ for all $r$. Hence $\e_j= \bigvee_{r=1}^n y_r \in I_*$, contradicting the fact that the top element does not belong to a maximal ideal.
In conclusion, $y\in I=\bigcup_{k=1}^n I_k$ for an arbitrary $y$, so that $I=A$.

($\Leftarrow$) Let $I$ be an ultramultideal. Let $x\land_i y\in I_i$ with $x\in I_r$ and $y\in I_k$ (with $r\neq i$ and $k\neq i$). As $q(\e_j,\e_k,\dots,\e_k,\e_i,\e_k,\dots,\e_k)=\e_k$ ($\e_i$ at position $i$), then by property (m2) of multideals we get $q(x,y,\dots,y,\e_i,y,\dots,y)=x\land_i y\in I_k$. Contradiction.
\end{proof}



\section*{Conclusion}
Boolean-like algebras have been introduced in \cite{first,BLPS18} as a generalisation of Boolean algebras to any finite number of truth values. Boolean-like algebras provide a new characterisation of primal varieties exhibiting a perfect symmetry of the values of the generator of the variety. This feature has been used in \cite{BLPS18} to define a $n$-valued propositional logic, where the truth values play perfectly symmetric roles, allowing an encoding of any tabular logic.
 
In this paper we have investigated the relationships between skew Boolean algebras and  Boolean-like algebras. 
We have shown that any $n$-dimensional Boolean-like algebra is a cluster of $n$ isomorphic right-handed skew Boolean algebras, and that the variety of  skew star algebras is  term equivalent to the variety of Boolean-like algebras. Moreover, we have got a representation theorem for right-handed skew Boolean algebras, and
developed a general theory of multideals for Boolean-like algebras.
  Several further works are worth mentioning:
  \begin{itemize}
       \item How the duality theory of skew BAs and  BAs
    are related to a possible duality theory of $n$BAs (a Stone-like topology on  ultramultideals).
\item Provide the proof theory of the logic $n$CL, whose equivalent algebraic semantics is the variety of $n$BAs.
\item Find a more satisfactory axiomatisation of skew star algebras.
\item Each skew BA living inside a $n$BA has a bottom element 0 and
  several maximal elements.  The construction could be made symmetric,
 by defining ``skew-like'' algebras having several minimal and several
  maximal elements. 
  \end{itemize}

\end{document}